\documentclass[letter,usenames,dvipsnames,12pt]{amsart}
\usepackage[utf8]{inputenc}
\usepackage{microtype}
    \usepackage{xcolor}
    \usepackage{hyperref}
    \hypersetup{colorlinks,
    citecolor=Violet,%
	linkcolor=RoyalBlue}

    \usepackage{amsmath,amsthm,amssymb,amsfonts,mathrsfs,mathtools}
    \usepackage{cleveref}
    \crefname{equation}{Eq.}{Eqs.}
    \usepackage{cases}
    \usepackage{enumerate}
    \newtheorem{theorem}{Theorem}
    \newtheorem{lemma}[theorem]{Lemma}
    \newtheorem{corollary}[theorem]{Corollary}
    
    \newtheorem{example}[theorem]{Example}
    \theoremstyle{definition}
    \newtheorem{definition}[theorem]{Definition}
    \theoremstyle{remark}
    \newtheorem{remark}[theorem]{Remark}
    \theoremstyle{remark}
    \newtheorem*{rmk}{Remark}
    \theoremstyle{theorem}
    \newtheorem*{thm}{Theorem}
    \theoremstyle{theorem}
    
    \theoremstyle{theorem}
    \newtheorem*{exmp}{Example}
    \DeclarePairedDelimiter\abs{\lvert}{\rvert}
    \DeclarePairedDelimiter\norm{\lVert}{\rVert}
    \DeclareMathOperator{\pb}{\mathbf{P}\mathopen{}}
    \DeclareMathOperator{\E}{\mathbf{E}\mathopen{}}

    \newcommand\Psub [1]{\pb_{\! #1}}

    \DeclareMathSymbol{\Z}{\mathbin}{AMSb}{"5A}
    \DeclareMathSymbol{\R}{\mathbin}{AMSb}{"52}
    \newcommand{\df}[1]{\,\mathrm{d}#1}
    \newcommand\st{\, \colon \;}
    \newcommand{\ue}{\mathrm{e}}

    \usepackage{extarrows} 

\usepackage[foot]{amsaddr} 

\begin{document}

    \title[Inverse Problems for Ergodicity of MC]{Inverse Problems for Ergodicity of Markov Chains}
    \author{Zhi-Feng Wei$^{1,2}$}
    \address{$^1$School of Mathematical Sciences and Lab.\@ Math.\@ Com.\@ Sys.\@, Beijing Normal University, Beijing 100875, China}
    \address{$^2$Department of Mathematics, Indiana University Bloomington, 831.\@ E 3rd St.\@, Bloomington, IN 47405, USA \texttt{zfwei@iu.edu}}



    \begin{abstract}
         For both continuous-time and discrete-time Markov Chains, we provide criteria for inverse problems of classical types of ergodicity: (ordinary) erogodicity, algebraic ergodicity, exponential ergodicity and strong ergodicity. Our criteria are in terms of the existence of solutions to inequalities involving the $Q$-matrix (or transition matrix $P$ in time-discrete case) of the process. Meanwhile, these criteria are applied to some examples and provide ``universal" treatment, including single birth processes and several multi-dimensional models.
    \end{abstract}

    \maketitle

\section{Introduction}\label{chp_intro}
        Criteria for various types of ergodicity by drift condition for Markov Chains have been studied extensively over the past decades, see \cite{hou1988,twd1981,mao2003,mao2004,wang2003}. According to these criteria, a solution to some inequality implies, for example, strong ergodicity of a Markov Chain. However, one may find it not that easy to make sure a process, for instance, not being strongly ergodic. In fact, the celebrated strong ergodicity criteria with drift condition reads as follows.
        \begin{thm}[\cite{hou1988,twd1981}]
            Let $Q$ be an irreducible regular $Q$-matrix and $H$ a non-empty finite subset of a countable state space $E$. Then the $Q$-process is strongly ergodic if and only if there exists a bounded solution $(y_i)_{i\in E}$ to inequality
                \begin{displaymath}
                    \sum_{j\in E}q_{ij}y_j\leqslant -1,\qquad i\notin H.
                \end{displaymath}
        \end{thm}
        \par
        If we are proving a $Q$-process is not strongly ergodic using this criterion, we have to show that there is no bounded solution to this inequality. Neverthless, this is not so practical. We intend to complement ergodicity criteria in this paper. For instance, can we assert non-strong ergodicity of a $Q$-process from some inequality and some of its solutions?
        \par
        Since we are dealing with ergodic properties, we assume processes considered are all recurrent without loss of generality. And we will deal with not only continuous-time but also discrete-time Markov Chains using exactly the same method.
        \par
        Consider an irreducible regular $Q$-matrix $Q=(q_{ij}\st i,j\in E)$ on a countable state space $E$ with transition probability matrix $P(t)=\bigl(p_{ij}(t)\bigr)_{t\geqslant0}$. Meanwhile, denote
        \begin{displaymath}
            q_i\coloneqq -q_{ii}=\sum_{j\neq i}q_{ij}<\infty,\qquad i\in E.
        \end{displaymath}
        We have the following ergodic notions.
            \begin{enumerate}[(1)]
                \item The $Q$-process is ergodic, if for each $i$, ${\norm[\big]{p_{i\cdot}(t)- \pi}_{\mathrm{Var}} \coloneqq  \sum_{j\in E}
                    \abs[\big]{p_{ij}(t)- \pi_j} \to 0}$ as $t \to \infty$.
                \item (algebraic ergodicity) The $Q$-process is $\ell$-ergodic for some integer $\ell\geqslant 1$, if for each $i,j\in E$, $\abs[\big]{p_{ij}(t)- \pi_j} = O\bigl(t^{-(\ell-1)}\bigr)$ as $t \to \infty$.
                \item The $Q$-process is exponentially ergodic, if for each $i,j\in E$, ${\abs[\big]{p_{ij}(t)- \pi_j}=O(\ue^{-\beta t})}$ as $t \to \infty$ for some $\beta>0$.
                \item The $Q$-process is strongly ergodic, if $\lim_{t\to\infty}\sup_{i\in E} \norm[\big]{ p_{i\cdot}(t)- \pi}_{\mathrm{Var}} = 0$.
            \end{enumerate}
        Note that we occasionally say a $Q$-process is $0$-ergodic when it is recurrent for ease of terminology. Also, we may say a $Q$-process is $1$-ergodic if it is ergodic.
        \par
        Set
        \begin{displaymath}
            \sigma_H \coloneqq \inf\bigl\{t\geqslant \eta_1\st X_t\in H\bigr\},\qquad H\subseteq E,
        \end{displaymath}
        where $(X_t)_{t\geqslant 0}$ is the $Q$-process and $\eta_1$ is the first jump time. There are probabilistic descriptions of above ergodic notions.
            \begin{enumerate}[(1)]
                \item The $Q$-process is ergodic if and only if (abbr.\@ iff) $\max_{i\in H}\E_i\mkern-1.5mu{\sigma_H}$ is finite for some (equivalently, for any) non-empty finite subset $H$ of $E$.
                \item (algebraic ergodicity) The $Q$-process is $\ell$-ergodic for some integer $\ell\geqslant 1$ iff $\max_{i\in H}\E_i\mkern-1.5mu\sigma_H^\ell$ is finite for some (equivalently, for any) non-empty finite subset $H$ of $E$.
                \item The $Q$-process is exponentially ergodic iff $\max_{i\in H}\E_i\mkern-1.5mu\ue^{\lambda\sigma_H}$ is finite for some positive $\lambda$ (with $\lambda<q_i, \forall i\in E$) and some (equivalently, for any) non-empty finite subset $H$ of $E$.
                \item The $Q$-process is strongly ergodic iff $\bigl(\E_i\mkern-1.5mu\sigma_H\bigr)_{i\notin H}$\vadjust{\kern2pt}%
                    is bounded for some (equivalently, for any) non-empty finite subset $H$ of $E$.
            \end{enumerate}
        \par
        Now, we declare our main results. Let $\Pi=\bigl(\Pi_{ij}\st i,j\in E\bigr)$ be the embedding chain of the $Q$-process, where we have
                    \begin{numcases}
                        {\Pi_{ij}=}
                            q_{ij}/q_i,\qquad \nonumber &$j\neq i$,\\
                            0, &$j=i$\nonumber.
                    \end{numcases}
        \begin{theorem}\label{in_erg_con}
            Let $Q$ be an irreducible regular $Q$-matrix and $H$ a non-empty finite subset of $E$. Then the $Q$-process is non-ergodic iff there is a sequence $\{y^{(n)}\}^{\infty}_{n=1}$, where $y^{(n)}=\bigl(y^{(n)}_i\bigr)_{i\in E}$ for each $n\geqslant 1$, and $\{y^{(n)}\}^{\infty}_{n=1}$ satisfies the following conditions:
                \begin{enumerate}[\upshape (1)]
                    \item for each $n\geqslant 1$, $\bigl(y^{(n)}_i\bigr)_{i\in E}$ satisfies $\sup_{i\in E} y^{(n)}_i<\infty$ and solves inequality
                         \begin{equation}\label{in_erg_con_eq}
                            y_i\leqslant\sum_{j\notin H}\Pi_{ij}y_j+\frac{1}{q_i},\qquad i\in E;
                         \end{equation}
                    \item $\sup_{n\geqslant 1} \max_{i\in H} y^{(n)}_i =\infty$ (or equivalently, $\varlimsup_{n\to \infty} \max_{i\in H} y^{(n)}_i =\infty$).
                \end{enumerate}
        \end{theorem}

        \begin{theorem}\label{in_serg_con}
            Let $Q$ be an irreducible regular $Q$-matrix and $H$ a non-empty finite subset of $E$. Then the $Q$-process is non-strongly ergodic iff there is a sequence $\{y^{(n)}\}^{\infty}_{n=1}$, where $y^{(n)}=\bigl(y^{(n)}_i\bigr)_{i\notin H}$ for each $n\geqslant 1$, and $\{y^{(n)}\}^{\infty}_{n=1}$ satisfies the following conditions:
                \begin{enumerate}[\upshape (1)]
                    \item for each $n\geqslant 1$, $\bigl(y^{(n)}_i\bigr)_{i\notin H}$ satisfies $\sup_{i\notin H} y^{(n)}_i<\infty$ and solves inequality
                        \begin{equation}\label{in_serg_con_eq}
                            y_i\leqslant\sum_{j\notin H}\Pi_{ij}y_j+\frac{1}{q_i},\qquad i\notin H;
                        \end{equation}
                    \item $\sup_{n\geqslant 1}\sup_{i\notin H} y^{(n)}_i =\infty$ (or equivalently, $\varlimsup_{n\to \infty} \sup_{i\notin H} y^{(n)}_i =\infty$).
                \end{enumerate}
        \end{theorem}

        \begin{rmk}
            Testing sequence in \Cref{in_erg_con,in_serg_con} need not be non-negative. Take \Cref{in_erg_con} for instance. Let $\{y^{(n)}\}^{\infty}_{n=1}$ be a sequence satisfying the conditions in \Cref{in_erg_con}. Then for each $n\geqslant 1$, $y^{(n)}=\bigl(y^{(n)}_i\bigr)_{i\in E}$ is a function on $E$. Here, $y^{(n)}$ is not required to be non-negative. We may even allow $\inf_{i\in E}y^{(n)}_i=-\infty$. However, $y^{(n)}$ should be a finite-valued function. In other words, for each $n\geqslant 1$ and $i\geqslant 1$, $y^{(n)}_i$ is a finite real number.
        \end{rmk}

        The following inverse problem criterion for algebraic ergodicity generalizes \Cref{in_erg_con}.
        \begin{theorem}\label{in_aerg_con}
            Let $Q$ be an irreducible regular $Q$-matrix and $H$ a non-empty finite subset of $E$. Suppose the $Q$-process is $\ell$-ergodic for some non-negative integer $\ell$, then the $Q$-process is not $(\ell+1)$-ergodic iff there is a sequence $\{y^{(n)}\}^{\infty}_{n=1}$, where $y^{(n)}=\bigl(y^{(n)}_i\bigr)_{i\in E}$ for each $n\geqslant 1$, and $\{y^{(n)}\}^{\infty}_{n=1}$ satisfies the following conditions:
                \begin{enumerate}[\upshape (1)]
                    \item for each $n\geqslant 1$, $\bigl(y^{(n)}_i\bigr)_{i\in E}$ satisfies $\sup_{i\in E} y^{(n)}_i<\infty$ and solves inequality
                         \begin{equation}\label{in_aerg_con_eq}
                            y_i\leqslant\sum_{j\notin H}\Pi_{ij}y_j+\frac{(\ell+1)}{q_i}\E_i\mkern-1.5mu\sigma_H^{\ell},\qquad i\in E;
                         \end{equation}
                    \item $\sup_{n\geqslant 1} \max_{i\in H} y^{(n)}_i =\infty$ (or equivalently, $\varlimsup_{n\to \infty} \max_{i\in H} y^{(n)}_i =\infty$).
                \end{enumerate}
        \end{theorem}

        \Cref{in_eerg_con} is a non-exponential ergodicity criterion for $Q$-processes.
        \begin{theorem}\label{in_eerg_con}
            Let $Q$ be an irreducible regular $Q$-matrix with $\inf_{i\in E} q_i>0$ and $H$ a non-empty finite subset of $E$. Then the $Q$-process is non-exponentially ergodic iff there is a sequence of positive numbers $\{\lambda_n\}_{n=1}^{\infty}$ and a sequence of functions $\{y^{(n)}\}^{\infty}_{n=1}$ on $E$  satisfying the following conditions:
                \begin{enumerate}[\upshape (1)]
                    \item $\lim_{n\to \infty}\lambda_n=0$;
                    \item for each $n\geqslant 1$, $\bigl(y^{(n)}_i\bigr)_{i\in E}$ is finitely supported and solves inequality
                         \begin{equation}\label{in_eerg_con_eq}
                            y_i^{(n)}\leqslant \frac{q_i}{q_i-\lambda_n}\sum_{j\notin H}\Pi_{ij}y^{(n)}_j+\frac{1}{q_i-\lambda_n},\qquad i\in E;
                         \end{equation}
                    \item $\sup_{n\geqslant 1} \max_{i\in H} y^{(n)}_i =\infty$ (or equivalently, $\varlimsup_{n\to \infty} \max_{i\in H} y^{(n)}_i =\infty$).
                \end{enumerate}
        \end{theorem}
        Although we need a sequence of testing functions in applications of above results, we can actually manufacture testing functions in batch. For example, one may consult the following interesting example and its proof in \Cref{sec_catastr}.
            \begin{exmp}
                Let $Q=(q_{ij})$ be a conservative $Q$-matrix on $E=\Z_+=\{0,1,2,\ldots\}$ with
                    \begin{numcases}
                        {q_{ij}=}
                            i+1,\quad \nonumber &if\/ $i\geqslant0,\enspace j=i+1$,\\
                            \alpha_i\geqslant0, &if\/ $i\geqslant1,\enspace j=0$,\nonumber\\
                            0, & other $i\neq j$\nonumber.
                    \end{numcases}
                Assume there are infinitely many non-zero $\alpha_i$, so $Q$ is irreducible. Then, the $Q$-process is non-exponentially ergodic if $ \lim_{i\to\infty}\alpha_i=0$.
    		\end{exmp}
        \par
        Brussel's model (see \cite{yanchen1986}) is a typical model of reaction-diffusion process with several species. Finite-dimensional Brussel's model is exponentially ergodic (cf.\@ \cite{chenjw1995}). In \Cref{chp_app}, we will demonstrate that it is non-strongly ergodic using \Cref{in_serg_con}, which was actually proved for the first time in \cite{wu2007} by comparison method.
        Comparison method works for Brussel's model but it is no longer available for more involved models like the following one. However, we can still deal with it using our drift criteria developed in this paper, see \Cref{chp_app} for further details.
        \begin{exmp}
             Let $S$ be a finite set, $E = (\Z_+)^S$ and $p(u, v)$ a transition probability matrix on $S$. We denote by $\theta \in E$ whose components are identically 0 and denote by $e_u \in E$ the unit vector whose component at site $u \in S$ is equal to 1 and other components at $v \neq u $ all equal 0. Define an irreducible $Q$-matrix $Q=\bigl(q(x,y)\st x,y\in E\bigr)$ as follows:
                \begin{numcases}
                    {q(x, y)=}
                    x(u)^{\gamma}, \nonumber &if\/ $y=x+e_u$,\enspace $x\neq\theta$, \\
                    1, \nonumber &if\/ $x=\theta$,\enspace$y=e_u$, \\
                    x(u)^{\gamma}, \nonumber &if\/ $y=x-e_u$, \\
                    x(u)p(u,v), \quad\nonumber &if\/ $y=x-e_u+e_v$,\enspace $v\neq u$,\\
                    0, & other $y\neq x$\nonumber,
                \end{numcases}
            and $q(x) = -q(x,x)=\sum_{y\neq x}q(x,y)$, where $x = \bigl(x(u)\st u\in S\bigr) \in E$. In \Cref{chp_app}, we will prove the following results:
                \begin{enumerate}[(1)]
                    \item when $\gamma\leqslant 2$, the $Q$-process is non-strongly ergodic;
                    \item when $\gamma\leqslant 1$, the $Q$-process is non-ergodic.
                \end{enumerate}
        \end{exmp}

        As for discrete time chains, we also have the following parallel criteria.
        \renewcommand{\thetheorem}{\ref{in_erg_con}$^\prime$}
        \addtocounter{theorem}{-1}
        \begin{theorem}\label{in_erg_dis}
            Let $P=(P_{ij})$ be an irreducible aperiodic transition matrix and $H$ a non-empty finite subset of $E$. Then the chain is non-ergodic iff there is a sequence $\{y^{(n)}\}^{\infty}_{n=1}$, where $y^{(n)}=\bigl(y^{(n)}_i\bigr)_{i\in E}$ for each $n\geqslant 1$, and $\{y^{(n)}\}^{\infty}_{n=1}$ satisfies the following conditions:
                \begin{enumerate}[\upshape (1)]
                    \item for each $n\geqslant 1$, $\bigl(y^{(n)}_i\bigr)_{i\in E}$ satisfies $\sup_{i\in E} y^{(n)}_i<\infty$ and solves inequality
                            \renewcommand{\theequation}{\ref{in_erg_con_eq}$^\prime$}
                            \addtocounter{equation}{-1}
                         \begin{equation}\label{in_erg_dis_eq}
                            y_i \leqslant \sum_{j\notin H}P_{ij}y_j+1,\qquad i\in E;
                         \end{equation}
                            \renewcommand{\theequation}{\arabic{equation}}
                    \item $\sup_{n\geqslant 1} \max_{i\in H} y^{(n)}_i =\infty$ (or equivalently, $\varlimsup_{n\to \infty} \max_{i\in H} y^{(n)}_i =\infty$).
                \end{enumerate}
         \end{theorem}
        \renewcommand{\thetheorem}{\arabic{theorem}}

        \renewcommand{\thetheorem}{\ref{in_serg_con}$^\prime$}
        \addtocounter{theorem}{-1}
        \begin{theorem}\label{in_serg_dis}
            Let $P=(P_{ij})$ be an irreducible aperiodic transition matrix and $H$ a non-empty finite subset of $E$. Then the chain is non-strongly ergodic iff there is
            $\{y^{(n)}\}^{\infty}_{n=1}$, where $y^{(n)}=\bigl(y^{(n)}_i\bigr)_{i\notin H}$ for each $n\geqslant 1$, and $\{y^{(n)}\}^{\infty}_{n=1}$ satisfies
            the following conditions:
                \begin{enumerate}[\upshape (1)]
                    \item for each $n\geqslant 1$, $\bigl(y^{(n)}_i\bigr)_{i\notin H}$ satisfies $\sup_{i\notin H} y^{(n)}_i<\infty$ and solves inequality
                            \renewcommand{\theequation}{\ref{in_serg_con_eq}$^\prime$}
                            \addtocounter{equation}{-1}
                        \begin{equation}\label{in_serg_dis_eq}
                            y_i \leqslant \sum_{j\notin H}P_{ij}y_j+1,\qquad i\notin H;
                        \end{equation}
                            \renewcommand{\theequation}{\arabic{equation}}
                    \item $\sup_{n\geqslant 1}\sup_{i\notin H} y^{(n)}_i =\infty$ (or equivalently, $\varlimsup_{n\to \infty} \sup_{i\notin H} y^{(n)}_i =\infty$).
                \end{enumerate}
         \end{theorem}
        \renewcommand{\thetheorem}{\arabic{theorem}}

        \renewcommand{\thetheorem}{\ref{in_aerg_con}$^\prime$}
        \addtocounter{theorem}{-1}
        \begin{theorem}\label{in_aerg_dis}
            Let $P=(P_{ij})$ be an irreducible aperiodic transition matrix and $H$ a non-empty finite subset of $E$. Suppose the chain is $\ell$-ergodic for some non-negative integer $\ell$, then the chain is not $(\ell+1)$-ergodic iff there is a sequence $\{y^{(n)}\}^{\infty}_{n=1}$, where $y^{(n)}=\bigl(y^{(n)}_i\bigr)_{i\in E}$ for each $n\geqslant 1$, and $\{y^{(n)}\}^{\infty}_{n=1}$ satisfies the following conditions:
                \begin{enumerate}[\upshape (1)]
                    \item for each $n\geqslant 1$, $\bigl(y^{(n)}_i\bigr)_{i\in E}$ satisfies $\sup_{i\in E} y^{(n)}_i<\infty$ and solves inequality
                            \renewcommand{\theequation}{\ref{in_aerg_con_eq}$^\prime$}
                            \addtocounter{equation}{-1}
                         \begin{equation}\label{in_aerg_dis_eq}
                            y_i\leqslant\sum_{j\notin H}P_{ij}y_j+\E_i\mkern-1.5mu\sigma_H^{\ell},\qquad i\in E;
                         \end{equation}
                            \renewcommand{\theequation}{\arabic{equation}}

                    \item $\sup_{n\geqslant 1} \max_{i\in H} y^{(n)}_i =\infty$ (or equivalently, $\varlimsup_{n\to \infty} \max_{i\in H} y^{(n)}_i =\infty$).
                \end{enumerate}
        \end{theorem}
        \renewcommand{\thetheorem}{\arabic{theorem}}
        The remainder of this paper is organized as follows. In \Cref{chp_prf}, we present proofs for our criteria. In \Cref{chp_sbp}, our criteria are applied to single birth processes. Some multi-dimensional models are treated in \Cref{chp_app}.


\section{Proofs of Criteria for Inverse Problems}\label{chp_prf}
    \subsection{Minimal Solution Theory Preparations}\label{sec_mini_soln}
    Our proofs are based on minimal solution theory. To begin, let's first recall promptly some useful results in minimal solution theory from \cite{hou1988,chen2004}.
    \par
    Let $E$ be an arbitrary non-empty set. Denote by $\mathscr{H}$ a set of mappings from $E$ to $\overline{\R}_+\coloneqq [0,+\infty]$: $\mathscr{H}$ contains constant 1 and is closed under non-negative linear combination and monotone increasing limit, where the order relation ``$\geqslant$'' in $\mathscr{H}$ is defined pointwise. Then, $\mathscr{H}$ is a convex cone. We say that $A\colon\mathscr{H}\to\mathscr{H}$ is a cone mapping if $A0=0$ and
            \begin{displaymath}
                A(c_1f_1+c_2f_2)=c_1Af_1+c_2Af_2,\qquad \text{for all }c_1, c_2\geqslant 0 \text{ and }f_1, f_2\in \mathscr{H}.
            \end{displaymath}
        Denote by $\mathscr{A}$ the set of all such mappings which also satisfy the following hypothesis:
            \begin{displaymath}
                \mathscr{H}\ni f_n\uparrow f \quad \text{implies}\quad  Af_n\uparrow Af.
            \end{displaymath}
			
        \begin{definition}
            Given $A\in\mathscr{A}$ and $g\in\mathscr{H}$. We say $f^{*}$ is a minimal non-negative solution (abbr.\@ minimal solution) to equation
                \begin{equation}\label{mini_eq}
                    f=Af+g,\qquad x\in E,
                \end{equation}
            if $f^*$ satisfies \Cref{mini_eq} and for any solution $\widetilde{f}\in\mathscr{H}$ to \Cref{mini_eq}, we have
                \begin{displaymath}
                    \widetilde{f}\geqslant f^*,\qquad x\in E.
                \end{displaymath}
        \end{definition}
		
        \begin{theorem}[\mbox{\cite[Theorem 2.2]{chen2004}}]\label{mini_uniq}
            The minimal solution to \Cref{mini_eq} always exists uniquely.
        \end{theorem}
		
        \begin{definition}
            Let $A, \widetilde{A}\in \mathscr{A}$ and $g, \widetilde{g}\in \mathscr{H}$ satisfy
                \begin{displaymath}
                   \widetilde{A}\geqslant A,\qquad \widetilde{g}\geqslant g.
                \end{displaymath}
            Then we call
                \begin{equation}\label{mini_ctrl}
                    \widetilde{f} \geqslant \widetilde{A}\widetilde{f}+\widetilde{g},\qquad x\in E
                \end{equation}
            a controlling equation of \Cref{mini_eq}.
        \end{definition}
		
        \begin{theorem}[\mbox{\cite[Theorem 2.6]{chen2004}, Comparison Principle}]\label{mini_cmprs}
            Let $f^*$ be the minimal solution to \Cref{mini_eq}. Then for any solution $\widetilde{f}$ to \Cref{mini_ctrl}, we have $\widetilde{f}\geqslant f^*$.
        \end{theorem}
		
        By \Cref{mini_uniq}, we may define a map
        \begin{displaymath}\begin{split}
             m_A\colon  \mathscr{H} &\to \mathscr{H},\\
                    g &\mapsto m_A g,
        \end{split}\end{displaymath}
        where $m_A g$ denotes the minimal solution to \Cref{mini_eq}.
        \begin{theorem}[\mbox{\cite[Theorem 2.7]{chen2004}}]\label{mini_apprx_org}
            $m_A$ is a cone mapping. For $\{A_n\}\subseteq \mathscr{A}$, $A_n \uparrow A$ and $\{g_n\}\subseteq \mathscr{H}$, $g_n\uparrow g$, we have $A\in \mathscr{A}, g\in\mathscr{H}$ and $m_{A_n}g_n\uparrow m_{A}g$.
        \end{theorem}
    The following minimal solution characterizations of moments of hitting times are essential for us to exploit minimal solution theory.
        \begin{theorem}[\mbox{\cite[Theorem 3.1]{mao2004}}]
            \label{con_mini_alge}
            For any $\ell\geqslant 1$, the moments of return times $\E_i\mkern-1.5mu\sigma_H,\E_i\mkern-1.5mu\sigma_H^2,\ldots,\E_i\mkern-1.5mu\sigma_H^\ell$ are inductively the minimal solution to the following $\ell$-family of systems for $0\leqslant n\leqslant \ell-1$,
            \begin{displaymath}
                x_i^{(n+1)}=\sum_{j\notin H}\Pi_{ij}x_j^{(n+1)}+\frac{(n+1)}{q_i}x_i^{(n)},\qquad i\in E,
            \end{displaymath}
            where $x_i^{(0)}=1\,(i\in E)$.
        \end{theorem}
		
        When $\ell=1$, \Cref{con_mini_alge} gives
        \begin{corollary}\label{con_min1}
            $(\E_i\mkern-1.5mu\sigma_H)_{i\in E}$ is the minimal solution to
            \begin{displaymath}
                x_i=\sum_{j\notin H}\Pi_{ij}x_j+\frac{1}{q_i},\qquad i\in E.
            \end{displaymath}
        \end{corollary}
        \begin{theorem}[\mbox{\cite[Theorem 4.48]{chen2004}}]\label{con_mini_exp}
            For a non-empty finite subset $H$ of $E$ and positive $\lambda$ with $\lambda<q_i$ for all $i\in E$, set
            \begin{displaymath}
                e_{iH}(\lambda)\coloneqq  \frac{1}{\lambda}\bigl(\E_i\mkern-1.5mu\ue^{\lambda\sigma_H}-1\bigr) = \int_0^\infty \ue^{\lambda t}\Psub{i}[\sigma_H>t]\df{t}
            \end{displaymath}
            for each $i\in E$ (cf.\@ \cite[Page 148, Equivalence of Theorems 4.45 and 4.44]{chen2004}). Then
            $\bigl(e_{iH}(\lambda)\bigr)_{i\in E}$ is the minimal solution to
            \begin{displaymath}
                x_i=\frac{q_i}{q_i-\lambda}\sum_{j\notin H}\Pi_{ij}x_j+\frac{1}{q_i-\lambda},\qquad i\in E.
            \end{displaymath}
        \end{theorem}

    To prove the criteria, we may assume $E=\{0, 1, 2, \ldots\}$ and $H=\{0\}$ without loss of generality. Since the proofs for discrete-time Markov Chains are similar with those for continuous-time Chains, we only give the proofs in time-continuous setup. One may easily prove time-discrete results using similar technic.
    \par
    Before proceeding further, let's briefly describe the main points in our proofs. Take non-ergodicity for instance. In order to prove the expectation of return time to the state $0$ is infinity, we first get a lower bound for the expectation of return time. Then a sequence of increasing lower bound implies the desired result.
    On another hand, finite approximation method would guarantee existence of an increasing sequence of lower bound and therefore necessity of our conditions.
    \par

	\subsection{Lower Bound for Polynomial Moments and Sufficiency}\label{sec_suffi}
        \begin{theorem}\label{dmc}
            Let $P=(P_{ij})$ be an irreducible conservative transition matrix on $E$. Then the chain is transient iff the inequality
                \begin{displaymath}
                    \sum_{j\geqslant 0}P_{ij} z_j\leqslant z_i,\qquad i\geqslant 1
                \end{displaymath}
            has a solution $z=(z_i)_{i\geqslant 0}$ satisfying
                \begin{displaymath}
                    -\infty<\inf_{i\geqslant 0} z_i <z_0.
                \end{displaymath}
        \end{theorem}

        As \Cref{dmc} is a slight modification of \cite[Theorem 4.25]{chen2004}, its proof would not be included here. One may also find a proof in \cite[Proposition 1.3]{martin2016}.

		\begin{lemma}\label{cmprs_a}
            Let $\ell$ be a non-negative integer and $Q$ an irreducible regular $Q$-matrix on $E$. Assume further inequality
				\begin{displaymath}
					y_i\leqslant\sum_{
                                    \begin{subarray}{c}
                                        j\geqslant 1\\
                                        j\neq i
                                    \end{subarray}}\frac{q_{ij}}{q_i}y_j+\frac{(\ell+1)}{q_i}\E_i\mkern-1.5mu\sigma_0^\ell,\qquad i\geqslant1
				\end{displaymath}
            has a finite solution $y=(y_i)_{i\geqslant 1}$ with $\sup_{i\geqslant1} y_i<\infty$. If the $Q$-process is $(\ell+1)$-ergodic, then we have
				\begin{displaymath}
					y_i\leqslant \E_i\mkern-1.5mu\sigma_0^{\ell+1}, \qquad i\geqslant 1.
				\end{displaymath}
		\end{lemma}

		\begin{proof}
            Since the $Q$-process is $(\ell+1)$-ergodic, $(\E_i\mkern-1.5mu\sigma_0^{\ell+1})_{i\geqslant 1}$ is finite and is the minimal non-negative solution to
				\begin{displaymath}
					x_i=\sum_{
                                    \begin{subarray}{c}
                                        j\geqslant 1\\
                                        j\neq i
                                    \end{subarray}}\frac{q_{ij}}{q_i}x_j+\frac{(\ell+1)}{q_i}\E_i\mkern-1.5mu\sigma_0^\ell,\qquad i\geqslant 1.
				\end{displaymath}
			Set
				\begin{numcases}
					{z_i=}
					0, \nonumber &$i=0$, \\
					\E_i\mkern-1.5mu\sigma_0^{\ell+1}-y_i,\qquad  \nonumber&$i\geqslant 1$.
				\end{numcases}
			Then $(z_i)_{i\geqslant0}$ satisfies
				\begin{displaymath}
				\left\{ \begin{aligned}
						 \sum_{j\geqslant 0}\Pi_{ij} z_j &\leqslant z_i,\qquad i\geqslant 1, \\
								  \inf_{i\geqslant 0} z_i&>-\infty.
				\end{aligned} \right.
				\end{displaymath}
            The $Q$-process is recurrent by our assumption, so is its embedding chain. Applying \Cref{dmc} to the embedding chain $\Pi=(\Pi_{ij})$, we arrive at the conclusion that
					\begin{displaymath}
						z_i\geqslant z_0,\qquad i\geqslant 1.
					\end{displaymath}
			In other words,
				\begin{displaymath}
					y_i\leqslant  \E_i\mkern-1.5mu\sigma_0^{\ell+1},\qquad i\geqslant 1.\qedhere
				\end{displaymath}
		\end{proof}

        \begin{remark}\label{upbound}
            The hypothesis ``$\sup_{i\geqslant1} y_i<\infty$'' cannot be removed from \Cref{cmprs_a}. In fact, we consider an ergodic $Q$-process, then $(\E_i\mkern-1.5mu\sigma_0)_{i\geqslant 1}$ is the minimal non-negative solution to
                \begin{equation}\label{aux}
                    x_i=\sum_{
                                    \begin{subarray}{c}
                                        j\geqslant 1\\
                                        j\neq i
                                    \end{subarray}}\frac{q_{ij}}{q_i}x_j+\frac{1}{q_i},\qquad i\geqslant 1.
                \end{equation}
            On another hand, fix an arbitrary $\varepsilon>0$, if we take
                \begin{displaymath}\begin{split}
                    x_1'&= \E_1\mkern-1.5mu\sigma_0 + \varepsilon,\\
                    x_i'&= x_1' \sum_{k=0}^{i-1}F_k^{(0)}-\sum_{k=0}^{i-1}d_k^{(0)},\qquad i\geqslant 2,
                \end{split}\end{displaymath}
            then $(x_i')_{i\geqslant 1}$ also solves \Cref{aux} (cf.\@ \cite{chenzhang2014}). Because the $Q$-process is assumed to be ergodic and thus recurrent, we have $\sum_{k=0}^\infty F_k^{(0)}=\infty$
            (cf.\@ \cite{chen2004, chenzhang2014}). Note that
                \begin{displaymath}
                    x_i'-\E_i\mkern-1.5mu\sigma_0= \varepsilon \sum_{k=0}^{i-1} F_k^{(0)},
                \end{displaymath}
            we may conclude that $(x_i')_{i\geqslant 1}$ is unbounded. Meanwhile, we have
                \begin{displaymath}
					\E_i\mkern-1.5mu\sigma_0 < x_i', \qquad i\geqslant 1.
				\end{displaymath}
        This implies that the condition ``$\sup_{i\geqslant1} y_i<\infty$'' cannot be removed.
        \end{remark}
        \par
        It is straightforward to write the time-discrete analogue of \Cref{cmprs_a} and we shall omit its proof.
        \renewcommand{\thetheorem}{\ref{cmprs_a}$^\prime$}
        \addtocounter{theorem}{-1}
        \begin{lemma}\label{cmprs_a_dis}
            Let $\ell$ be a non-negative integer and $P$ an irreducible aperiodic transition matrix on $E$. Assume further inequality
                \begin{displaymath}
                    y_i\leqslant\sum_{j\geqslant 1}P_{ij}y_j+\E_i\mkern-1.5mu\sigma_0^\ell,\qquad i\geqslant1
                \end{displaymath}
            has a finite solution $y=(y_i)_{i\geqslant 1}$ with $\sup_{i\geqslant1} y_i<\infty$.  If the chain is $(\ell+1)$-ergodic, then we have
                \begin{displaymath}
                    y_i\leqslant \E_i\mkern-1.5mu\sigma_0^{\ell+1}, \qquad i\geqslant 1.
                \end{displaymath}
            \par \vspace{-1.2\baselineskip}
                \qed
        \end{lemma}
        \renewcommand{\thetheorem}{\arabic{theorem}}
        \begin{proof}[Proof of sufficiency of $\Cref{in_aerg_con}$]
            If the $Q$-process is $(\ell+1)$-ergodic, by \Cref{con_mini_alge} and \Cref{cmprs_a}, for each $n\geqslant 1$,
                \begin{displaymath}
                    y^{(n)}_0\leqslant \sum_{
                                    \begin{subarray}{c}
                                        j\geqslant 1\\
                                    \end{subarray}}\frac{q_{0j}}{q_0}y^{(n)}_j+\frac{(\ell+1)}{q_i}\E_0\mkern-1.5mu\sigma_0^\ell \leqslant
                    \sum_{
                                    \begin{subarray}{c}
                                        j\geqslant 1\\
                                    \end{subarray}}\frac{q_{0j}}{q_0}\E_j\mkern-1.5mu\sigma_0^{\ell+1}+\frac{(\ell+1)}{q_i}\E_0\mkern-1.5mu\sigma_0^\ell=\E_0\mkern-1.5mu\sigma_0^{\ell+1}.
                \end{displaymath}
            It follows that
            \begin{displaymath}
                    \infty=\sup_{n\geqslant1}y^{(n)}_0 \leqslant \E_0\mkern-1.5mu\sigma_0^{\ell+1} <\infty,
                \end{displaymath}
            a contradiction.
        \end{proof}

        \begin{proof}[Proof of sufficiency of $\Cref{in_serg_con}$]
            It suffices to prove \Cref{in_serg_con} when the $Q$-process is ergodic. By \Cref{cmprs_a} with $\ell=0$, we have
                \begin{displaymath}
                    y^{(n)}_i\leqslant \E_i\mkern-1.5mu\sigma_0, \qquad i\geqslant 1,\enspace n\geqslant 1.
                \end{displaymath}
            Consequently,
            \begin{displaymath}
                    \infty=\sup_{n\geqslant1}\sup_{i\geqslant1}y^{(n)}_i\leqslant \sup_{i\geqslant1}\E_i\mkern-1.5mu\sigma_0.
                \end{displaymath}
            Thus the $Q$-process is non-strongly ergodic. Our proof is now complete.
        \end{proof}

    \subsection{Approximation for Polynomial Moments and Necessity}\label{sec_neces}
        Let $\ell$ be a fixed non-negative integer. To prove necessity of \Cref{in_aerg_con,in_serg_con}, we consider truncated equations for each $n\geqslant 1$:
            \addtocounter{equation}{1}
            \begin{equation}\label{apprxa_eq}
                x_i=\sum_{
                                    \begin{subarray}{c}
                                        1\leqslant j\leqslant n\\
                                        j\neq i
                                    \end{subarray}}\frac{q_{ij}}{q_i}x_j+\frac{(\ell+1)}{q_i}\E_i\mkern-1.5mu\sigma_0^\ell,\qquad 1\leqslant i\leqslant n.\tag{\theequation.n}
            \end{equation}
        Denote the minimal non-negative solution to \Cref{apprxa_eq} as
        \begin{displaymath}
            x^{(n)}=\bigl(x^{(n)}_i,\ 1\leqslant i\leqslant n\bigr).
        \end{displaymath}
        Also, we set $ M_n=\max_{1\leqslant i\leqslant n}x^{(n)}_i$.
		
        \begin{lemma}\label{apprxa}
            If the $Q$-process is $\ell$-erogdic, then we have the following assertions:
            \begin{enumerate}[\upshape (1)]
                \item $M_n$ is finite for each $n\geqslant 1$;\label{apprxa:1}
                \item $\E_i\mkern-1.5mu\sigma_0^{\ell+1}= \lim_{n\to\infty} {\hskip -7pt} \uparrow x^{(n)}_i$ for each $i\geqslant 1$, and $(M_n)_{n\geqslant1}$ is increasing;\label{apprxa:2}
                \item $(M_n)_{n\geqslant1}$ is bounded iff\/ $(\E_i\mkern-1.5mu\sigma_0^{\ell+1})_{i\geqslant 1}$ is bounded;\label{apprxa:3}
                \item pick $\ell=0$, then it follows from \eqref{apprxa:3} that the $Q$-process is non-strongly ergodic iff\/ $\sup_{n\geqslant1} M_n=\infty$.\label{apprxa:4}
            \end{enumerate}
        \end{lemma}
		
        \begin{proof}
            a) Since the $Q$-process is $\ell$-ergodic, we may pick a positive constant
                    \begin{displaymath}
                        C_n = (\ell+1) \max_{1\leqslant i\leqslant n}\E_i\mkern-1.5mu\sigma_0^\ell+1.
                    \end{displaymath}
                Now consider inequality
                    \begin{displaymath}
                        x_i \geqslant \sum_{
                                    \begin{subarray}{c}
                                        1\leqslant j\leqslant n\\
                                        j\neq i
                                    \end{subarray}}\frac{q_{ij}}{q_i}x_j+\frac{C_n}{q_i},\qquad 1\leqslant i\leqslant n.
                    \end{displaymath}
                Introducing a change of variable $\widetilde{x}_i=\frac{x_i}{C_n}$, we have the following equivalent form of the above inequality:
                    \addtocounter{equation}{1}
                    \begin{equation}\label{apprxa_eq_aux}\tag{\theequation.n}
                        \widetilde{x}_i \geqslant  \sum_{
                                    \begin{subarray}{c}
                                        1\leqslant j\leqslant n\\
                                        j\neq i
                                    \end{subarray}}\frac{q_{ij}}{q_i}\widetilde{x}_j+\frac{1}{q_i},\qquad 1\leqslant i\leqslant n.
                    \end{equation}
                By \Cref{con_min1}, the minimal solution to \Cref{apprxa_eq_aux} is the expectation of return time to state 0 of the $Q^{(n)}$-process and is therefore finite, where $Q^{(n)}$ has the following form:
                        \begin{displaymath}       
                            Q^{(n)}=
                            \left(                 
                              \begin{array}{ccccc}   
                                -n\quad &{1\ \ } &{1\ \ } &\cdots &{1}\\  
                                q_{10}+\sum_{k=n+1}^{\infty}q_{1,k}\quad &{q_{11}\ } &{q_{12} } &\cdots &{q_{1n}}\\
                                \vdots & \vdots & \vdots &\vdots &\vdots\\
                                q_{n0}+\sum_{k=n+1}^{\infty}q_{n,k}\quad &{q_{n1}\ } &{q_{n2}\ } &\cdots &{q_{nn}}\\
                              \end{array}
                            \right)_{(n+1)\times(n+1).}               
                        \end{displaymath}
                Now by \Cref{mini_cmprs}, $M_n$ is finite.
            \par
            b) By \Cref{con_mini_alge}, $\bigl(\E_i\mkern-1.5mu\sigma_0^{\ell+1}\bigr)_{i\geqslant 1}$ is the minimal solution to
                \begin{displaymath}
                    x_i=\sum_{
                                    \begin{subarray}{c}
                                        1\leqslant j\leqslant n\\
                                        j\neq i
                                    \end{subarray}}\frac{q_{ij}}{q_i}x_j+\frac{(\ell+1)}{q_i}\E_i\mkern-1.5mu\sigma_0^\ell,\qquad i\geqslant 1.
                \end{displaymath}
            Exploiting \Cref{mini_apprx_org}, we obtain the second assertion.
            \par
            c) Some trivial manipulation leads to the other two assertions. We omit the details.
        \end{proof}

        \begin{proof}[Proof of necessity of $\Cref{in_aerg_con}$]
            Suppose the $Q$-process is not $(\ell+1)$-ergodic. Set
                \begin{numcases}
                    {y^{(n)}_i=}
                    \sum_{
                                    \begin{subarray}{c}
                                        j\geqslant 1\\
                                    \end{subarray}}\frac{q_{0j}}{q_0}x^{(n)}_j+\frac{(\ell+1)}{q_0}\E_0\mkern-1.5mu\sigma_0^\ell, \quad\nonumber &$i=0$, \\
                    x^{(n)}_i, \nonumber &$1\leqslant i \leqslant n$, \\
                    0, &$i\geqslant n+1$.\nonumber
                \end{numcases}
            By the monotone convergence theorem and \Cref{con_mini_alge},
                \begin{displaymath}\begin{split}
                    \lim_{n\to \infty} y^{(n)}_0 &= \lim_{n\to \infty} \sum_{
                                    \begin{subarray}{c}
                                        j\geqslant 1\\
                                    \end{subarray}}\frac{q_{0j}}{q_0}y^{(n)}_j+\frac{(\ell+1)}{q_0}\E_0\mkern-1.5mu\sigma_0^\ell\\
                    &=
                    \sum_{
                                    \begin{subarray}{c}
                                        j\geqslant 1\\
                                    \end{subarray}}\frac{q_{0j}}{q_0}\E_j\mkern-1.5mu\sigma_0^{\ell+1}+\frac{(\ell+1)}{q_0}\E_0\mkern-1.5mu\sigma_0^\ell\\
                    &=\E_0\mkern-1.5mu\sigma_0^{\ell+1} =\infty.
                \end{split}\end{displaymath}
            Now it is easy to check that $\{y^{(n)}\}^{\infty}_{n=1}$ with $y^{(n)}=\bigl(y^{(n)}_i\bigr)_{i\in E}$ is a required sequence. Necessity of \Cref{in_aerg_con} is proved.
        \end{proof}

        \begin{proof}[Proof of necessity of $\Cref{in_serg_con}$]
            Assume the $Q$-process is non-strongly ergodic. We pick $\ell=0$ in \Cref{apprxa} and set
                \begin{numcases}
                    {y^{(n)}_i=}
                    x^{(n)}_i, \quad\nonumber &$1\leqslant i \leqslant n$, \\
                    0, &$i\geqslant n+1$.\nonumber
                \end{numcases}
            Then $\{y^{(n)}\}^{\infty}_{n=1}$ is a sequence required in \Cref{in_serg_con}. In fact, we may easily deduce that for each $n\geqslant 1$, $y^{(n)}$ solves \Cref{in_serg_con_eq}. Meanwhile, for each $n\geqslant 1$, we have
            \begin{displaymath}
                \sup_{i\geqslant1} y^{(n)}_i=M_n<\infty.
            \end{displaymath}
            By the last assertion of \Cref{apprxa}, $\sup_{n\geqslant1} M_n=\infty$. Therefore,
            \begin{displaymath}
                \sup_{n\geqslant 1} \sup_{i\geqslant1} y^{(n)}_i =\sup_{n\geqslant 1}M_n=\infty.
            \end{displaymath}
        Hence we prove necessity of \Cref{in_serg_con}.
        \end{proof}
		
    \subsection{Proof of \Cref{in_eerg_con}}\label{sec_eergpf}
        Now, we prove \Cref{in_eerg_con}, non-exponential ergodicity criteria. Since we are discussing exponential ergodicity in this subsection, we assume the process is ergodic without loss of generality. Our idea for proof of \Cref{in_eerg_con} is similar with that of \Cref{in_aerg_con,in_serg_con} but technical details here are different and more complex. Briefly speaking, we first use \Cref{mini_ctrl_in} to get a lower bound for exponential moment of return time. On another hand, we use finite approximation to prove the necessity.
        \par
        First, using the notation in \Cref{sec_mini_soln}, we have the following two useful lemmas.
        \begin{theorem}[\mbox{\cite[Theorem 2.10]{chen2004}}]\label{mini_iter}
            Given an arbitrary non-negative $\widetilde{f}^{(0)}$ satisfying $0\leqslant \widetilde{f}^{(0)} \leqslant pf^*$ for some non-negative number $p$, set
            \begin{displaymath}
                    \widetilde{f}^{(n+1)} = A\widetilde{f}^{(n)}+g,\qquad n\geqslant 0.
            \end{displaymath}
            Then we have $\widetilde{f}^{(n)} \to f^*\,(n\to\infty)$.
        \end{theorem}

        \begin{lemma}\label{mini_ctrl_in}
            Let $f^*$ be the minimal solution to \Cref{mini_eq} and $\widetilde{f}$ be a non-negative function satisfying
                \begin{equation}\label{mini_ctrl_in_eq}
                    \widetilde{f} \leqslant A\widetilde{f}+ g ,\qquad x\in E.
                \end{equation}
            If $\widetilde{f}\leqslant pf^*$ for some non-negative number $p$, then $\widetilde{f}\leqslant f^*$.
        \end{lemma}
        \begin{proof}
            Assume $p>1$ without loss of generality. Define
                \begin{displaymath}\begin{split}
                    \widetilde{f}^{(0)}&=\widetilde{f},\\
                    \widetilde{f}^{(n+1)}&=A\widetilde{f}^{(n)}+g,\qquad n\geqslant 0.
                \end{split}\end{displaymath}
            We claim
                 \begin{displaymath}
                    \widetilde{f}^{(n)}\uparrow f^*,\qquad \text{as }n\to \infty.
                 \end{displaymath}
            In fact, by \Cref{mini_iter}, we have
                \begin{displaymath}
                    \widetilde{f}^{(n)}\to f^*,\qquad \text{as }n\to \infty.
                 \end{displaymath}
            So we need only show the monotonicity. According to \Cref{mini_ctrl_in_eq},
                \begin{displaymath}
                    \widetilde{f}^{(0)} \leqslant A\widetilde{f}^{(0)}+ g = \widetilde{f}^{(1)}.
                \end{displaymath}
            Now if $\widetilde{f}^{(n)} \leqslant \widetilde{f}^{(n+1)}$ for some $n\geqslant 0$, then
                \begin{displaymath}
                    \widetilde{f}^{(n+1)} = A\widetilde{f}^{(n)}+ g \leqslant A\widetilde{f}^{(n+1)}+ g=
                    \widetilde{f}^{(n+2)}.
                \end{displaymath}
            So the monotonicity holds by induction.
            It follows immediately that
                \begin{displaymath}
                    \widetilde{f}= \widetilde{f}^{(0)} \leqslant f^*.
                \end{displaymath}
            \Cref{mini_ctrl_in} is proved.
        \end{proof}

        \par
                Let $Q$ be a $Q$-matrix on $E$ with $\inf_{i\in E}q_i>0$. Fix an integer $N\geqslant 1$ and consider $Q$-matrix on finite states
                        \begin{displaymath}
                            Q^{(N)}=
                            \left(                 
                              \begin{array}{ccccc}   
                                -N\quad &{1\ } & {1\ } &\cdots &{1}\\  
                                q_{10}+\sum_{k=N+1}^{\infty}q_{1,k}\quad &{q_{11}\ } &{q_{12}\ } &\cdots &{q_{1N}}\\
                                \vdots & \vdots & \vdots &\vdots &\vdots\\
                                q_{N0}+\sum_{k=N+1}^{\infty}q_{N,k}\quad &{q_{N1}\ } &{q_{N2}\ } &\cdots &{q_{NN}}\\
                              \end{array}
                            \right)_{(N+1)\times(N+1).}               
                        \end{displaymath}
                Meanwhile, we consider the following equation for $\lambda\in\bigl(0, \inf_{i\in E}q_i\bigr)$:
                    \begin{equation}\label{eloc}
                        x_i=\frac{q_i}{q_i-\lambda}\sum_{
                                    \begin{subarray}{c}
                                        1\leqslant j\leqslant N\\
                                        j\neq i
                                    \end{subarray}}\frac{q_{ij}}{q_i}x_j + \frac{1}{q_i-\lambda}, \qquad 0\leqslant i \leqslant N.
                    \end{equation}
                Denote the minimal solution to \Cref{eloc} as $\bigl(x_i^{(\lambda,N)}, 0\leqslant i \leqslant N\bigr)$. Then by \Cref{mini_apprx_org}, we have
                    \begin{displaymath}
                        x_0^{(\lambda,N)} \uparrow e_{00}(\lambda),\qquad \text{as }N\to \infty.
                    \end{displaymath}
                Also, we set $\lambda'=\frac{1}{2}\inf_{i\in E}q_i$.
        \begin{lemma}\label{in_eerg_prep}
            \begin{enumerate}[\upshape (1)]
                \item Assume the $Q$-process is non-exponentially ergodic, then
                        \begin{displaymath}
                            \lim_{N\to\infty} {\hskip -7pt} \uparrow x^{(\lambda', N)}_0=e_{00}(\lambda')=\infty.
                        \end{displaymath}
                \item If\/ $x_0^{(\widetilde{\lambda},N)}$ is finite for some $\widetilde{\lambda}\in\bigl(0, \inf_{i\in E}q_i\bigr)$, then for some $\widehat{\lambda}\in\bigl(\widetilde{\lambda}, \inf_{i\in E}q_i\bigr)$, $x_0^{(\widehat{\lambda},N)}$ is finite.
                \item If\/ $x_0^{(\widetilde{\lambda},N)}<\infty$ for some $\widetilde{\lambda}\in\bigl(0, \inf_{i\in E}q_i\bigr)$, then $x_0^{(\lambda,N)}$ is continuous at $\widetilde{\lambda}$ as a function of $\lambda$.
                \item If\/ $x_0^{(\widetilde{\lambda},N)}=\infty$ for some $\widetilde{\lambda}\in\bigl(0, \inf_{i\in E}q_i\bigr)$, then
                    \begin{displaymath}\begin{split}
                        &\lim_{\lambda\uparrow \widetilde{\lambda}}x_0^{(\lambda,N)}=\infty,\\
                        &x_0^{(\lambda,N)}=\infty,\qquad \lambda>\widetilde{\lambda}.
                    \end{split}\end{displaymath}
                    In other words, $x_0^{(\lambda,N)}$ is continuous at $\widetilde{\lambda}$ as an extended real-valued function.
                \item For any fixed integer $N\geqslant1$,
                    \begin{displaymath}
                        \lim_{\lambda\downarrow 0} x_0^{(\lambda,N)}\leqslant \E_0\mkern-1.5mu\sigma_0<\infty.
                    \end{displaymath}
            \end{enumerate}
        \end{lemma}
        \begin{proof}
            a) The first assertion is a direct inference of \Cref{mini_apprx_org} and non-exponential ergodicity.
            \par
            b) By \Cref{eloc}, $\bigl(2x_i^{(\widetilde{\lambda},N)}, 0\leqslant i \leqslant N\bigr)$ is a finite solution to
                    \begin{displaymath}
                        x_i=\frac{q_i}{q_i-\widetilde{\lambda}}\sum_{
                                    \begin{subarray}{c}
                                        1\leqslant j\leqslant N\\
                                        j\neq i
                                    \end{subarray}}\frac{q_{ij}}{q_i}x_j + \frac{2}{q_i-\widetilde{\lambda}}, \qquad 0\leqslant i \leqslant N.
                    \end{displaymath}
                So it satisfies
                    \begin{displaymath}
                        x_i>\frac{q_i}{q_i-\widetilde{\lambda}}\sum_{
                                    \begin{subarray}{c}
                                        1\leqslant j\leqslant N\\
                                        j\neq i
                                    \end{subarray}}\frac{q_{ij}}{q_i}x_j + \frac{1}{q_i-\widetilde{\lambda}}, \qquad 0\leqslant i \leqslant N.
                    \end{displaymath}
                Consequently, $\bigl(2x_i^{(\widetilde{\lambda},N)}, 0\leqslant i \leqslant N\bigr)$ also satisfies
                    \begin{displaymath}
                        x_i>\frac{q_i}{q_i-\widehat{\lambda}}\sum_{
                                    \begin{subarray}{c}
                                        1\leqslant j\leqslant N\\
                                        j\neq i
                                    \end{subarray}}\frac{q_{ij}}{q_i}x_j + \frac{1}{q_i-\widehat{\lambda}}, \qquad 0\leqslant i \leqslant N,
                    \end{displaymath}
                for some $\widehat{\lambda}$ slightly larger than $\widetilde{\lambda}$. Now by \Cref{mini_cmprs}, $x_0^{(\widehat{\lambda},N)}$ is finite.
            \par
            c) By the second assertion, to prove the third one, we need only prove $x_0^{(\lambda,N)}$ is continuous on the interval $(0,\widetilde{\lambda}]$. Because
                \begin{displaymath}
                     x_i^{(\lambda,N)}=\frac{1}{\lambda}\bigl(\E_i\mkern-1.5mu^{(Q^{(N)})}\ue^{\lambda\sigma_0}-1\bigr), \qquad 1\leqslant i \leqslant N,
                \end{displaymath}
            $x_i^{(\lambda,N)}\,(i=1,2,\ldots,N)$ is continuous on the interval $(0,\widetilde{\lambda}]$ by the Lebesgue dominated convergence theorem. Furthermore, $x_0^{(\lambda,N)}$ is continuous on the interval according to equality
                \begin{displaymath}
                    x_0=\frac{q_0}{q_0-\lambda}\sum_{1\leqslant j \leqslant N}\frac{q_{0j}}{q_0}x_j + \frac{1}{q_0-\lambda}.
                \end{displaymath}
            \par
            d) The fourth assertion is obvious according to above discussions.
            \par
            e) Now we prove the last assertion. Since the $Q$-process is assumed to be ergodic, $\E_0\mkern-1.5mu\sigma_0<\infty$. We need only illustrate
                \begin{displaymath}
                    \lim_{\lambda\downarrow 0} x_0^{(\lambda,N)}\leqslant \E_0\mkern-1.5mu\sigma_0.
                \end{displaymath}
            By the proof of ``Equivalence of Theorems 4.45 and 4.44" in \cite[Page 148]{chen2004}, we have
                \begin{displaymath}
                    x_i^{(\lambda,N)}=\int_{0}^{\infty}\ue^{\lambda t}\Psub{i}^{(Q^{(N)})}[\sigma_0>t]\df{t},\qquad i\geqslant 1.
                \end{displaymath}
            Because the $Q^{(N)}$-process, as a process on finite state space, must be exponentially ergodic, the Lebesgue dominated convergence theorem gives
                \begin{displaymath}
                    \lim_{\lambda\downarrow 0} x_i^{(\lambda,N)}
                    =\int_{0}^{\infty}\Psub{i}^{(Q^{(N)})}[\sigma_0>t]\df{t}
                    = \E_i^{(Q^{(N)})}\mkern-1.5mu\sigma_0
                    \leqslant \E_i\mkern-1.5mu\sigma_0,\qquad i\geqslant 1,
                \end{displaymath}
            where the last inequality is by \Cref{mini_apprx_org,con_min1}. Furthermore, by \Cref{eloc},
                \begin{displaymath}\begin{split}
                    \lim_{\lambda\downarrow 0} x_0^{(\lambda,N)}&=
                    \lim_{\lambda\downarrow 0}\frac{q_0}{q_0-\lambda}\sum_{1\leqslant j \leqslant N}\frac{q_{0j}}{q_0}x^{(\lambda,N)}_j + \lim_{\lambda\downarrow 0}\frac{1}{q_0-\lambda}\\
                    &=\sum_{1\leqslant j \leqslant N}\frac{q_{0j}}{q_0}\E_j^{(Q^{(N)})}\mkern-1.5mu\sigma_0 + \frac{1}{q_0}\\
                    &\leqslant \sum_{j \geqslant 1}\frac{q_{0j}}{q_0}\E_j\mkern-1.5mu\sigma_0 + \frac{1}{q_0}=\E_0\mkern-1.5mu\sigma_0.
                \end{split}\end{displaymath}
            Therefore, the last assertion holds.
        \end{proof}

        \begin{corollary}\label{func}
            For each $N\geqslant 1$, $x_0^{(\lambda,N)}$ is an extended real-valued continuous function as a funtion of $\lambda$ on interval $(0,\lambda']$.\qed
        \end{corollary}
        \begin{proof}[Proof of necessity of $\Cref{in_eerg_con}$]
            For each positive integer $n\leqslant\E_0\mkern-1.5mu\sigma_0$, we define $y^{(n)}_i\equiv 0\,(i\in E)$ and $\lambda_n=\lambda'$. And for each $n>\E_0\mkern-1.5mu\sigma_0$, we now construct $y^{(n)}=\bigl(y^{(n)}_i\bigr)_{i\in E}$ and $\lambda_n$ satisfying
                \begin{displaymath}
                    y^{(n)}_0\geqslant n,\quad\lambda_n\leqslant\frac{1}{n}.
                \end{displaymath}
            In fact, by the first assertion of \Cref{in_eerg_prep}, we may pick a large $N_n$ such that
                \begin{displaymath}
                    x^{(\lambda', N_n)}_0\geqslant n.
                \end{displaymath}
            Then for each $N\geqslant N_n$,
                \begin{displaymath}
                    x^{(\lambda', N)}_0\geqslant n.
                \end{displaymath}
            Furthermore, by \Cref{func} and the last assertion of \Cref{in_eerg_prep}, for each $N\geqslant N_n$, there exists $\lambda(n,N)\in (0,\lambda']$ such that
                \begin{displaymath}
                    x^{(\lambda(n,N), N)}_0= n.
                \end{displaymath}
            For ease of notation, we write $ c=\inf_{N\geqslant N_n}\lambda(n,N)$. Now, we claim $c=0$.\\
            Otherwise if $c>0$, we have
                \begin{displaymath}
                    e_{00}(c)=\lim_{N\to\infty}x_0^{(c, N)}\leqslant n,
                \end{displaymath}
            contradicting non-exponential ergodicity.\\
            Consequently, we may pick some $\lambda(n,\widetilde{N}_n)\leqslant\frac{1}{n}$ and denote it as $\lambda_n$. Then we have $\lambda_n\leqslant\frac{1}{n}$ and $x^{(\lambda_n, \widetilde{N}_n)}_0=n$. Set
                \begin{numcases}
                    {y^{(n)}_i=}
                    x^{(\lambda_n,\widetilde{N}_n)}_i, \quad\nonumber &$0\leqslant i \leqslant \widetilde{N}_n$, \\
                    0, &$i\geqslant \widetilde{N}_n+1$.\nonumber
                \end{numcases}
           It is now straightforward to verify that $\{\lambda_n\}_{n=1}^\infty$ and $\{y^{(n)}\}_{n=1}^\infty$ are the desired sequences. Necessity of our condition follows immediately.
        \end{proof}
        \begin{proof}[Proof of sufficiency of $\Cref{in_eerg_con}$]
        a) We first demonstrate
            \begin{displaymath}
                y^{(n)}_0 \leqslant e_{00}(\lambda_n),\qquad n\geqslant 1.
            \end{displaymath}
        In fact, since $\bigl(y^{(n)}\bigr)_{i\in E}$ is finitely supported for each $n\geqslant 1$, we may pick $N_n$ such that
            \begin{displaymath}
                y_i^{(n)}\leqslant \frac{q_i}{q_i-\lambda_n}
                \sum_{
                                    \begin{subarray}{c}
                                        1\leqslant j\leqslant N_n\\
                                        j\neq i
                                    \end{subarray}}\frac{q_{ij}}{q_i}y^{(n)}_j+\frac{1}{q_i-\lambda_n},\qquad 1\leqslant i\leqslant N_n.
            \end{displaymath}
        At the same time, denote the minimal solution of
            \begin{displaymath}
                x_i=\frac{q_i}{q_i-\lambda_n}\sum_{
                                    \begin{subarray}{c}
                                        1\leqslant j\leqslant N_n\\
                                        j\neq i
                                    \end{subarray}}\frac{q_{ij}}{q_i}x_j + \frac{1}{q_i-\lambda_n}, \qquad 1\leqslant i \leqslant N_n
            \end{displaymath}
        as $\bigl(x_i^{(\lambda_n,N_n)}, 1\leqslant i \leqslant N_n\bigr)$, which is positive. Then
        by \Cref{mini_apprx_org} and \Cref{mini_ctrl_in},
            \begin{displaymath}
                y^{(n)}_i \leqslant x_i^{(\lambda_n,N_n)} \leqslant e_{i0}(\lambda_n),\qquad 1\leqslant i\leqslant N_n.
            \end{displaymath}
        It follows that
            \begin{displaymath}\begin{split}
                y_0^{(n)}
                &\leqslant \frac{q_0}{q_0-\lambda_n}\sum_{1\leqslant j\leqslant N_n}\frac{q_{0j}}{q_0}y^{(n)}_j+\frac{1}{q_0-\lambda_n}\\
                &\leqslant \frac{q_0}{q_0-\lambda_n}\sum_{1\leqslant j\leqslant N_n}\frac{q_{0j}}{q_0}e_{j0}(\lambda_n)+\frac{1}{q_0-\lambda_n}\\
                &\leqslant \frac{q_0}{q_0-\lambda_n}\sum_{j\geqslant 1}\frac{q_{0j}}{q_0}e_{j0}(\lambda_n)+\frac{1}{q_0-\lambda_n}
                =e_{00}(\lambda_n),
            \end{split}\end{displaymath}
        where the last equality is by \Cref{con_mini_exp}. This is exactly the desired inequality.
        \par
        b) For an arbitrary $\lambda>0$, when $\lambda_n<\lambda$,
            \begin{displaymath}
                y^{(n)}_0 \leqslant e_{00}(\lambda_n) \leqslant e_{00}(\lambda).
            \end{displaymath}
         Consequently,
            \begin{displaymath}
                \infty=\varlimsup_{n\to \infty}y_0^{(n)}\leqslant e_{00}(\lambda).
            \end{displaymath}
         It turns out that $\E_0\mkern-1.5mu\ue^{\lambda\sigma_0}=\infty\,(\lambda>0)$. So the $Q$-process is non-exponentially ergodic. Sufficiency of \Cref{in_eerg_con} is proved.
        \end{proof}
\section{Applications to Single Birth Processes}\label{chp_sbp}
    \subsection{Explicit Criteria for Single Birth Processes: Alternative Proofs}\label{sec_explct}
        Explicit and computable criteria for ergodicity and strong ergodicity of single birth processes have been studied in \cite{yanchen1986,zhang2001}, respectively. In this section, we present alternative proofs (of the necessity parts) for these explicit criteria.
        \par
        Let $Q$ be an irreducible regular single birth $Q$-matrix on state space $E=\Z_+=\{0,1,2,\ldots\}$. We have
            \begin{displaymath}\begin{split}
                &q_{i,i+1}>0,\\
                &q_{i,i+j}=0,\qquad i\geqslant0,\enspace j\geqslant2.
            \end{split}\end{displaymath}
        Define $q_n^{(k)}=\sum_{j=0}^{k}q_{nj}$ for $0\leqslant k<n\,(k,n\geqslant0)$ and
            \begin{displaymath}
                F_n^{(n)}=1,\quad F_n^{(i)}=\frac{1}{q_{n,n+1}}\sum_{k=i}^{n-1}q_n^{(k)}F_k^{(i)}\,(0\leqslant i\leqslant n),
            \end{displaymath}
            \begin{equation}\label{def_d}
                d_0=0,\quad  d_n=\frac{1}{q_{n,n+1}}\Bigl(1+\sum_{k=0}^{n-1}q_n^{(k)}d_k\Bigr) = \sum_{k=1}^n\frac{F_n^{(k)}}{q_{k,k+1}}\,(n\geqslant1).
            \end{equation}
        Also, we define
            \begin{displaymath}
                d=\sup_{k\geqslant 0}\frac{\sum_{n=0}^k d_n}{\sum_{n=0}^k F_n^{(0)}}.
            \end{displaymath}
            It is well-known that the $Q$-process is recurrent iff $\sum_{n=0}^\infty F_n^{(0)}=\infty$
            (cf.\@ \cite{chen2004, chenzhang2014}).
            \par
            To give alternative proofs for explicit ergodicity criteria for single birth processes, we first make some preparations.
        \begin{lemma}\label{eu_solu}
            Let $Q$ be an irreducible regular single birth $Q$-matrix and $N$ a positive integer. We investigate the following (truncated) equation:
            \begin{equation}\label{sglbth_erg_eq}
                    x_i = \sum_{
                                    \begin{subarray}{c}
                                        1\leqslant j\leqslant N\\
                                        j\neq i
                                    \end{subarray}}\frac{q_{ij}}{q_i}x_j+\frac{1}{q_i},\qquad 1\leqslant i\leqslant N.
            \end{equation}
            \begin{enumerate}[\upshape (1)]
                \item  \Cref{sglbth_erg_eq} has a unique solution, denoted as $\bigl(x^{(N)}_1,x^{(N)}_2,\ldots,x^{(N)}_N\bigr)$.
                \item   We have a recurrence relation:
                        \begin{displaymath}
                            x^{(N)}_k = x^{(N)}_1 \sum_{n=0}^{k-1}F_n^{(0)}- \sum_{n=0}^{k-1}d_n,\qquad 1\leqslant k\leqslant N.
                        \end{displaymath}
                \item The unique solution is positive.
                \item $\varlimsup_{N\to \infty}x^{(N)}_1 \geqslant d$.
            \end{enumerate}
        \end{lemma}
        \begin{proof}
            a) \Cref{sglbth_erg_eq} has the following equivalent form:
                    \begin{displaymath}
                        \sum_{j=1}^N q_{ij}x_j=-1,\qquad 1\leqslant i\leqslant N.
                    \end{displaymath}
                To prove regularity of this linear system, we need only prove the following homogeneous equation
                    \begin{equation}\label{var_sglbth_erg_eq}
                        \sum_{j=1}^N q_{ij}x_j=0,\qquad 1\leqslant i\leqslant N
                    \end{equation}
                has only trivial solution.
                \par
                Otherwise, if \Cref{var_sglbth_erg_eq} had a non-trivial solution  $(\overline{x}_1,\overline{x}_2,\ldots,\overline{x}_N)$, assume $\overline{x}_1\geqslant 0$ without loss of generality. We claim $\overline{x}_1\leqslant \overline{x}_2$. Since if $\overline{x}_1> \overline{x}_2$, \Cref{var_sglbth_erg_eq} with $i=1$ leads to
                    \begin{displaymath}
                        0=q_{11}\overline{x}_1+q_{12}\overline{x}_2< q_{11}\overline{x}_1+q_{12}\overline{x}_1\leqslant 0,
                    \end{displaymath}
                a contradiction. So we obtain $\overline{x}_1\leqslant \overline{x}_2$. Furthermore, we may proceed to prove that $\overline{x}_k\leqslant \overline{x}_{k+1}$ using similar arguments for $k=2,3,\ldots,N-1$. That is
                    \begin{displaymath}
                        \overline{x}_1\leqslant \overline{x}_2\leqslant\cdots\leqslant\overline{x}_N.
                    \end{displaymath}
                Since the solution is non-trivial, we have $\overline{x}_N>0$. Therefore,
                    \begin{displaymath}\begin{split}
                        0&=q_{N1}\overline{x}_1+q_{N2}\overline{x}_2+\cdots+q_{N,N-1}\overline{x}_{N-1}+q_{N,N}\overline{x}_N\\
                        &\leqslant (q_{N1}+q_{N2}+\cdots+q_{N,N-1}+q_{N,N})\overline{x}_N<0,
                    \end{split}\end{displaymath}
                a contradiction. So \Cref{var_sglbth_erg_eq} has only trivial solution. In this way, we prove the first assertion.
            \par
            b) To prove the second assertion, we mimic the proof of \cite[Lemma 2.1]{zhang2001}. Define
                \begin{displaymath}
                    v_0=x^{(N)}_1,\ v_n=x^{(N)}_{n+1}-x^{(N)}_n,\qquad 1\leqslant n \leqslant N-1.
                \end{displaymath}
            From \Cref{sglbth_erg_eq}, we easily derive that
                \begin{displaymath}
                    v_n=\frac{1}{q_{n,n+1}}\Bigl(\sum_{k=0}^{n-1}q_n^{(k)}v_k-1\Bigr),\qquad 1\leqslant n\leqslant N-1.
                \end{displaymath}
            By induction, $v_n=v_0 F_n^{(0)}-d_n$ for $0\leqslant n\leqslant N-1$. Our assertion follows immediately.
            \par
            c) If $x^{(N)}_i=\min_{1\leqslant k\leqslant N}x^{(N)}_k \leqslant 0$, then
                \begin{displaymath}\begin{split}
                    -1&=\sum_{j=1}^N q_{ij}x^{(N)}_j=\sum_{j=1}^{i-1} q_{ij}\bigl(x^{(N)}_j-x^{(N)}_i\bigr) -q_{i0}x^{(N)}_i\\
                      &\ \ \ +(1-\delta_{i,N}) q_{i,i+1}\bigl(x^{(N)}_{i+1}-x^{(N)}_i\bigr)- \delta_{i,N}q_{i,i+1}x^{(N)}_i \geqslant 0,
                \end{split}\end{displaymath}
            where $\delta$ is the Kronecker delta. This contradiction infers that the unique solution is positive.
            \par
            d) By the second assertion and the positiveness of the solution, we have
                \begin{displaymath}
                    x^{(N)}_1 > \max_{1\leqslant k\leqslant N} \frac{\sum_{n=0}^{k-1} d_n}{\sum_{n=0}^{k-1} F_n^{(0)}}.
                \end{displaymath}
               So the last assertion follows immediately.
        \end{proof}
		\par
        We are now in position to present our alternative proofs for explicit criteria of single birth processes.
		\par
        The following ergodicity criterion is due to Shi-Jian Yan and Mu-Fa Chen \cite{yanchen1986}. Here, proof for sufficiency is picked from \cite{yanchen1986} for completeness.
        \begin{theorem}\label{explct_erg}
            Let $Q$ be a regular single birth $Q$-matrix, then the $Q$-process is ergodic iff\/ $d<\infty$.
        \end{theorem}
        \begin{proof}
			a) When $d<\infty$, we define
				\begin{displaymath}
					y_0=0,\ y_k=\sum_{n=0}^{k-1}\bigl(F_n^{(0)}d - d_n\bigr),\qquad k\geqslant 1.
				\end{displaymath}
                Then $(y_i)_{i\geqslant0}$ satisfies the condition of \cite[Theorem 4.45(1)]{chen2004} with  $H=\{0\}$. So the $Q$-process is ergodic when $d<\infty$.
			 \par
			b) When $d=\infty$, for each $N\geqslant1$, we define
				\begin{displaymath}
					y^{(N)}_0=x^{(N)}_1+\frac{1}{q_1},\quad
					y^{(N)}_i= x^{(N)}_i\,(1\leqslant i\leqslant N),\quad
					y^{(N)}_i= 0\,(i\geqslant N+1).
				\end{displaymath}
                Because $\varlimsup_{N\to \infty}x^{(N)}_1 \geqslant d=\infty$, it can be easily seen that the conditions of \Cref{in_erg_con} are satisfied by the sequences $\{y^{(N)}\}^{\infty}_{N=1}$ and $H=\{0\}$ . So the $Q$-process is non-ergodic if $d=\infty$.
        \end{proof}

		\par
        The following strong ergodicity criterion is due to Yu-Hui Zhang \cite{zhang2001}.
        \begin{theorem}\label{explct_serg}
            Let $Q$ be a regular single birth $Q$-matrix, then the $Q$-process is strongly ergodic iff\/ $\sup_{k\geqslant0}\sum_{j=0}^k \bigl(F_j^{(0)}d-d_j\bigr)<\infty$.
        \end{theorem}
        \begin{proof}
            We assume the process is ergodic without loss of generality. In light of \Cref{explct_erg}, $d<\infty$ equivalently.
			\par
            a) When $\sup_{k\geqslant0}\sum_{j=0}^k \bigl(F_j^{(0)}d-d_j\bigr)<\infty$, we define
				\begin{displaymath}\begin{split}
					&y_0=0,\\
                    &y_k=\sum_{n=0}^{k-1}\bigl(F_n^{(0)}d - d_n\bigr),\qquad k\geqslant 1.
				\end{split}\end{displaymath}
            Then $(y_i)_{i\geqslant0}$ satisfies the condition of \cite[Teorem 4.45(3)]{chen2004} with  $H=\{0\}$. So the $Q$-process is strongly ergodic. This proof of sufficiency is not original but picked from \cite{zhang2001}.
			\par
            b) When $\sup_{k\geqslant0}\sum_{j=0}^k \bigl(F_j^{(0)}d-d_j\bigr)=\infty$, for each $N\geqslant1$, we define
				\begin{displaymath}
					y^{(N)}_i= x^{(N)}_i\,(1\leqslant i\leqslant N),
                    \quad
					y^{(N)}_i= 0\,(i\geqslant N+1).
				\end{displaymath}
                It is obvious that $ \sup_{i\geqslant 1}y^{(N)}_i<\infty$ for each $N\geqslant1$. We now prove that $ \varlimsup_{N\to \infty}\sup_{i\geqslant 1}y^{(N)}_i=\infty$. In fact, for an arbitrary $k\geqslant1$,
					\begin{displaymath}\begin{split}
						\varlimsup_{N\to \infty}\sup_{i\geqslant 1}y^{(N)}_i
						&\geqslant
						\varlimsup_{N\to \infty}x^{(N)}_k
                        =\varlimsup_{N\to \infty}\sum_{n=0}^{k-1}\bigl(F_n^{(0)}x^{(N)}_1 - d_n\bigr)\\
						&\geqslant
						\sum_{n=0}^{k-1}\bigl(F_n^{(0)}d - d_n\bigr).
					\end{split}\end{displaymath}
                Taking supremum with respect to $k$ on both sides, we obtain
                \begin{displaymath}
                    \varlimsup_{N\to \infty}\sup_{i\geqslant 1}y^{(N)}_i=\infty.
                \end{displaymath}
                The conditions of \Cref{in_serg_con} are satisfied by the sequences $\{y^{(N)}\}^{\infty}_{N=1}$ and $H=\{0\}$. So the $Q$-process is non-strongly ergodic.
        \end{proof}

    \subsection{A Special Class of Single Birth Processes}\label{sec_catastr}
        In this section, we study conservative single birth $Q$-matrix $Q=(q_{ij})$ with
            \begin{numcases}
                {q_{ij}=}
                    i+1,\quad \nonumber &if\/ $i\geqslant0,\enspace j=i+1$,\\
                    \alpha_i\geqslant0, &if\/ $i\geqslant1,\enspace j=0$,\nonumber\\
                    0, & other $i\neq j$\nonumber.
            \end{numcases}
        Assume there are infinitely many non-zero $\alpha_i$, so $Q$ is irreducible. The following illuminating example is a catalyst for this part.
        \begin{example}
            It is obvious that the $Q$-process is unique for arbitrary $\{\alpha_i\}_{i=1}^{\infty}$.
            \begin{enumerate}[\upshape (1)]
                \item If\/ $\alpha_i=\frac{1}{i^\gamma}$ for sufficiently large $i$, the $Q$-process is transient for $\gamma>0$.
                \item If\/ $\alpha_i=\frac{1}{\log^{\gamma}i}$ for sufficiently large $i$,
                    \begin{enumerate}[\upshape(a)]
                        \item the $Q$-process is transient for $\gamma>1$;
                        \item the $Q$-process is null recurrent for $\gamma=1$;
                        \item the $Q$-process is ergodic but non-exponentially ergodic for $\gamma \in(0,1)$.
                    \end{enumerate}
                \item If\/ $\alpha_i=\frac{1}{(\log\log i)^\gamma}$ for sufficiently large $i$, the $Q$-process is ergodic but non-exponentially ergodic for $\gamma>0$.\\
                \item If\/
                        \begin{numcases}
                            {\alpha_i=}
                                \tfrac{1}{i},\quad \nonumber &$i$ is an odd positive integer,\\
                                1, \nonumber &$i$ is an even positive integer,
                        \end{numcases}
                    the $Q$-process is strongly ergodic.
                \item The $Q$-process is strongly ergodic if\/ $\alpha_i\equiv1\,(i\geqslant1)$.
            \end{enumerate}
        \end{example}
        This example will be demonstrated via the following propositions.
            \begin{lemma}\label{cata_dis_rec}
                \begin{enumerate}[\upshape (1)]
                  \item Let $P=(P_{ij})$ be an irreducible conservative transition matrix on $\Z_+=\{0,1,2,\ldots\}$ with
        				\begin{numcases}
        					{P_{ij}=}
        						p_i,\quad \nonumber &if\/ $i\geqslant0,\enspace j=i+1$,\\
        						1-p_i, &if\/ $i\geqslant0,\enspace j=0$,\nonumber\\
        						0, & other $i,j\geqslant 0$\nonumber.
        				\end{numcases}
        			Then $P$ is recurrent iff\/ $\prod_{i=0}^{\infty}p_i=0$.
                  \item The $Q$-process mentioned above is recurrent iff\/ $\sum_{i=1}^{\infty}\frac{\alpha_i}{i}=\infty$.
                \end{enumerate}
    		\end{lemma}
    		\begin{proof}
    			a) By Theorems 4.24 and 4.25 in~\cite{chen2004}, we consider equation
    				\begin{equation}\label{cata_dis_rec_eq}
    					(1-p_i) y_0+p_i y_{i+1}=y_i,\qquad i\geqslant 1.
    				\end{equation}
    			Setting $y_0=0$, we obtain a recurrence relation:
    				\begin{displaymath}
    					y_{i+1}=\frac{1}{p_i}y_i,\qquad i\geqslant 1.
    				\end{displaymath}
                So \Cref{cata_dis_rec_eq} has a compact solution (non-constant bounded solution, respectively) if $\prod_{i=0}^{\infty}\frac{1}{p_i}=\infty$ ($<\infty$, respectively). This completes our proof.
                b) By the first assertion, the $Q$-process is recurrent iff $\prod_{i=1}^{\infty}\frac{i+1}{i+1+\alpha_i}=0$. Note that
        			\begin{displaymath}
                        \prod_{i=1}^{\infty}\frac{i+1}{i+1+\alpha_i}=0
                        \quad (\Longleftrightarrow )\quad
                        \sum_{i=1}^{\infty}\frac{\alpha_i}{i}=\infty.
        			\end{displaymath}
                The second assertion follows immediately.
    		\end{proof}

    	    \begin{lemma}\label{cata_in_eerg}
    			The $Q$-process is non-exponentially ergodic if\/ $\lim_{i\to\infty}\alpha_i=0$.
    		\end{lemma}
    		\begin{proof}
                First, we deal with a special case: $\{\alpha_i\}_{i=1}^{\infty}$ is monotonically decreasing. For a fixed $n\geqslant 1$, we set
    				\begin{numcases}
    					{y^{(n)}_i=}
    						1/\alpha_i,\quad \nonumber &$1\leqslant i \leqslant n$, \\
    						1/\alpha_n, &$i\geqslant n+1$.\nonumber
    				\end{numcases}
                It is straightforward to check that $y^{(n)}=\bigl(y^{(n)}_i\bigr)_{i\geqslant1}$ satisfies
    				\begin{displaymath}
                        (i+1+\alpha_i)y^{(n)}_i\leqslant(i+1)y^{(n)}_{i+1}+1,\qquad i\geqslant 1.
    				\end{displaymath}
                So $\{ y^{(n)}\}^{\infty}_{n=1}$ is a sequence satisfying all conditions of \Cref{in_serg_con}. The $Q$-process is non-strongly ergodic.
    			\par
                Now if the $Q$-process is exponentially ergodic, by \Cref{con_mini_exp}, the following \Cref{exp_rec} has a finite non-negative solution $(x_i)_{i\geqslant1}$ for some $\lambda\in (0,1)$.
    				\begin{equation}\label{exp_rec}
                        x_i=\frac{i+1}{i+1+\alpha_i-\lambda}x_{i+1}+\frac{1}{i+1+\alpha_i-\lambda},\quad i\geqslant 1.
    				\end{equation}
    			Equivalently,
    				\begin{displaymath}
    					x_{i+1}=\frac{i+1+\alpha_i-\lambda}{i+1}x_i-\frac{1}{i+1},\quad i\geqslant 1.
    				\end{displaymath}
                Because $\lim_{i\to\infty}\alpha_i=0$, we have $x_{i+1}\leqslant x_i$ for sufficiently large $i$.
                So $(x_i)_{i\geqslant1}$ is bounded. Consequently, $\bigl(\frac{1}{\lambda}(\E_i\mkern-1.5mu \ue^{\lambda\sigma_0}-1)\bigr)_{i\geqslant 1}$ is bounded since it is the minimal non-negative solution to \Cref{exp_rec}. Hence $\bigl(\E_i\mkern-1.5mu \ue^{\lambda\sigma_0}\bigr)_{i\geqslant 1}$ is bounded and so is $(\E_i\mkern-1.5mu \sigma_0)_{i\geqslant 1}$. The $Q$-process is thus strongly ergodic. This is impossible. The $Q$-process is therefore non-exponentially ergodic.
    			\par
                In general case where $\{\alpha_i\}_{i=1}^{\infty}$ may not be monotonically decreasing, we define conservative $\widetilde{Q}=(\widetilde{q}_{ij})$:
    				\begin{numcases}
    					{\widetilde{q}_{ij}=}
    						i+1,\quad \nonumber &if\/ $i\geqslant0,\enspace j=i+1$,\\
    						\sup_{k\geqslant i}\alpha_k, &if\/ $i\geqslant1,\enspace j=0$,\nonumber\\
    						0, & other $i\neq j$\nonumber.
    				\end{numcases}
    			Because
    			\begin{displaymath}
                    \lim_{i\to\infty}\bigl(\sup_{k\geqslant i}\alpha_k\bigr)=\varlimsup_{i\to\infty}\alpha_i=\lim_{i\to\infty}\alpha_i=0,
    			\end{displaymath}
                the $\widetilde{Q}$-process is non-exponentially ergodic according to above discussions. Consequently, the $Q$-process is non-exponentially ergodic
                by comparison. Our proof is now complete.
    		\end{proof}
            The above proof is based on \Cref{in_serg_con}, we may give a more direct proof using \Cref{in_eerg_con}.
            \begin{proof}[Alternative Proof of $\Cref{cata_in_eerg}$]
                Without loss of generality, we assume that $ \lim_{i\to\infty} {\hskip -5pt}\downarrow\alpha_i = 0$. First, set
    				\begin{displaymath}
    					\lambda_n=\frac{1}{n+1},\qquad n\geqslant 1.
    				\end{displaymath}
    			For each fixed positive integer $n$, by \Cref{in_eerg_con}, we consider
    				\begin{displaymath}\begin{split}
                        y_0^{(n)} &\leqslant \frac{1}{1-\lambda_n}y_1^{(n)} +\frac{1}{1-\lambda_n},\\
                        y_i^{(n)} &\leqslant \frac{i+1}{i+1+\alpha_i-\lambda_n}y_{i+1}^{(n)} +\frac{1}{i+1+\alpha_i-\lambda_n},\qquad i\geqslant 1.
    				\end{split}\end{displaymath}
                Introducing a change of variable $d_i^{(n)}=y_{i+1}^{(n)}-y_i^{(n)}\,(i\geqslant 0)$,
                the above inequality is transformed  into
                    \begin{displaymath}\begin{split}
                        d_0^{(n)} &\geqslant -\lambda_ny_0^{(n)}-1,\\
                        d_i^{(n)} &\geqslant \frac{1}{i+1}(\alpha_i-\lambda_n)y_i^{(n)}-\frac{1}{i+1},\qquad i\geqslant 1.
                    \end{split}\end{displaymath}
                Put $y_0^{(n)}=n$. As $\lim_{i\to\infty} {\hskip -5pt}\downarrow\alpha_i = 0$, there exists $M_1$ such that
                    \begin{displaymath}\begin{split}
                        \alpha_i &\geqslant \lambda_n,\qquad 1 \leqslant i \leqslant M_1-1,\\
                        \alpha_i &< \lambda_n,\qquad  i \geqslant M_1.
                    \end{split}\end{displaymath}
                If we place
                    \begin{displaymath}\begin{split}
                        d_0^{(n)} &= 0,\\
                        d_i^{(n)} &=\frac{\alpha_i-\lambda_n}{i+1}y_i^{(n)},\qquad 1 \leqslant i \leqslant M_1-1,
                    \end{split}\end{displaymath}
                then
                    \begin{displaymath}
                        n=y_0^{(n)}=y_1^{(n)}\leqslant y_2^{(n)}\leqslant \cdots \leqslant y_{M_1}^{(n)}.
                    \end{displaymath}
                Furthermore, we may pick $M_2> M_1$ such that
                    \begin{displaymath}
                        y_{M_1}^{(n)}-\frac{1}{M_1+1}-\cdots-\frac{1}{M_2}\geqslant 0,
                    \end{displaymath}
                    \begin{displaymath}
                        y_{M_1}^{(n)}-\frac{1}{M_1+1}-\cdots-\frac{1}{M_2}-\frac{1}{M_2+1}<0.
                    \end{displaymath}
                Meanwhile, let
                    \begin{displaymath}\begin{split}
                        d^{(n)}_k&=-\frac{1}{k+1},\qquad M_1\leqslant k \leqslant M_2-1,\\
                        d^{(n)}_{M_2}&=-y^{(n)}_{M_2},\\
                        d^{(n)}_k&=0,\qquad k\geqslant M_2+1.\\
                    \end{split}\end{displaymath}
                Thus $y^{(n)}_k=0\,(k>M_2)$.
                \par
                Now, one may check that $\{\lambda_n\}_{n=1}^{\infty}$ coupled with $\{y^{(n)}\}^{\infty}_{n=1}$ are sequences satisfying conditions in \Cref{in_eerg_con}. The $Q$-process is non-exponentially ergodic.
            \end{proof}

            \begin{corollary}\label{temp_flag}
    			Let $\alpha_i=\frac{1}{\log^{\gamma}i}\,(i\geqslant3)$.
    				\begin{enumerate}[\upshape (1)]
    					\item The $Q$-process is ergodic for $\gamma\in (0,1)$.
    					\item The $Q$-process is null recurrent for $\gamma=1$.
    				\end{enumerate}
    		\end{corollary}
    		\begin{proof}
    			a) When $\gamma\in (0,1)$, we set
    					$y_i=\log^{2\gamma}i\,(i\geqslant 3)$.
    			Then for sufficiently large $i$,
    				\begin{displaymath}
    					(i+1+\alpha_i)y_i \geqslant (i+1)y_{i+1}+1.
    				\end{displaymath}
    				In fact, for large $i$, by Lagrange mean value theorem,
    				\begin{displaymath}
    					(i+1)\bigl(\log^{2\gamma}(i+1)-\log^{2\gamma}i\bigr)
    					\leqslant
    					2\gamma\frac{i+1}{i}\log^{2\gamma-1}(i+1) \leqslant \log^{\gamma}i-1.
    				\end{displaymath}
    				Thus, the $Q$-process is ergodic for $\gamma\in (0,1)$ by \cite[Teorem 4.45(1)]{chen2004}.
                b) To obtain the second assertion, we try to exploit \Cref{explct_erg}. Using the O'Stolz theorem and the explicit expression of $F_i^{k)}$ in \cite[Example 8.2]{chenzhang2014}, we have
    				\begin{displaymath}\begin{split}
    					d&=\sup_{i\geqslant 0}\frac{\sum_{k=0}^{i}d_k}{\sum_{k=0}^{i}F_k^{(0)}}
    					\geqslant \lim_{i\to \infty}\frac{\sum_{k=0}^{i}d_k}{\sum_{k=0}^{i}F_k^{(0)}}
                        =\lim_{i\to \infty}\frac{d_i}{F_i^{(0)}}\\
                        &=\lim_{i\to \infty}\frac{\sum_{k=1}^i\frac{F_i^{(k)}}{q_{k,k+1}}}{F_i^{(0)}}\geqslant \lim_{i\to\infty} \sum_{k=1}^{i-1}\frac{1}{(k+1)\prod_{\ell=1}^k(1+\frac{\alpha_l}{\ell+1})}\\
                        &=\sum_{k=1}^{\infty}\frac{1}{(k+1)\prod_{\ell=1}^k(1+\frac{\alpha_l}{\ell+1})}.\\
    				\end{split}\end{displaymath}
                    Now, by Kummer's test, one may see $d=\infty$ for $\alpha_i=\frac{1}{\log i}\,(i\geqslant3)$. The $Q$-process is therefore non-ergodic.
            \end{proof}
            \begin{lemma}
                Let $Q$ be an irreducible regular $Q$-matrix and assume the $Q$-process is recurrent. If\/ $\inf_{i\geqslant 1} q_{i0}>0$, then the $Q$-process is strongly ergodic.
    		\end{lemma}
    		\begin{proof}
                Take $c\in (0,\inf_{i\geqslant 1} q_{i0})$, then
    				\begin{displaymath}
                        \frac{1}{c}\geqslant \frac{1}{c}+\frac{1-\frac{q_{i0}}{c}}{q_i}=\frac{1}{c}(1-\frac{q_{i0}}{q_i})+\frac{1}{q_i}
                        =\sum_{
                                        \begin{subarray}{c}
                                            j\geqslant 1\\
                                            j\neq i
                                        \end{subarray}}\frac{q_{ij}}{q_i}\frac{1}{c}+\frac{1}{q_i},\qquad i\geqslant 1.
    				\end{displaymath}
    			So the $Q$-process is strongly ergodic by \cite[Teorem 4.45(3)]{chen2004}.
    		\end{proof}
            \begin{lemma}
                Suppose $\{\alpha_i\}_{i=1}^{\infty}$ has a subsequence $\{\alpha_{i_k}\}_{k=1}^{\infty}$ satisfying
    			\begin{displaymath}
    				\inf_{k\geqslant 1}\alpha_{i_k}>0,\quad
    				\sup_{k\geqslant 1}\frac{i_{k+1}}{i_k}<\infty,\quad
    				\sum_{k=1}^\infty \frac{1}{i_k}=\infty.
    			\end{displaymath}
    			Then the $Q$-process is strongly ergodic.
    		\end{lemma}
    		\begin{proof}
                For ease of notation, we write $i_0=0$. Define conservative $\widetilde{Q}=(\widetilde{q}_{ij})$:
    				\begin{numcases}
    					{\widetilde{q}_{ij}=}
    						i+1,\qquad &if\/ $i=i_k,\enspace j=i_{k+1}$ for some $k\geqslant0$,\nonumber \\
    						i+1, &if\/ $i_k<i<i_{k+1}$ for some $k\geqslant0$,\enspace $j=i+1$,\nonumber \\
                            c\coloneqq  \tfrac{1}{2}\inf_{k\geqslant 1}\alpha_{i_k}, &if\/ $i=i_k$ for some $k\geqslant1$,\enspace $j=0$,\nonumber\\
    						0, & other $i\neq j$\nonumber.
    				\end{numcases}
                It is easy to see that $\{i_k\}_{k=0}^{\infty}$ is an irreducible subclass of $\widetilde{Q}$. Note that $\{i_k\}_{k=0}^{\infty}$ is also a recurrent subclass of the $\widetilde{Q}$-process since $\sum_{k=1}^\infty \frac{1}{i_k}=\infty$ (This is easy to illustrate using~\Cref{cata_dis_rec}). Because
    				\begin{displaymath}
    					\frac{1}{c}= \frac{i_k+1}{i_k+1+c}\cdot\frac{1}{c}+\frac{1}{i_k+1+c},\qquad k\geqslant1,
    				\end{displaymath}
                $\{i_k\}_{k=0}^{\infty}$ is furthermore a strongly ergodic subclass according to \cite[Teorem 4.45(3)]{chen2004}. Since $\sup_{k\geqslant 1}\frac{i_{k+1}}{i_k}<\infty$ implies
    				\begin{displaymath}
    					\sup_{k\geqslant 0}\Bigl(\frac{1}{i_k+2}+\frac{1}{i_k+3}+\cdots+\frac{1}{i_{k+1}}\Bigr)
    					\leqslant \sup_{k\geqslant 1}\frac{i_{k+1}-i_k-1}{i_k+1} < \infty,
    				\end{displaymath}
                    exploiting
    				\begin{displaymath}
                        \E_i^{(\widetilde{Q})}\mkern-1.5mu\sigma_0
                        = \E_{i+1}^{(\widetilde{Q})}\mkern-1.5mu\sigma_0 + \frac{1}{i+1}, \qquad i_k < i < i_{k+1},\enspace k\geqslant0,
    			   \end{displaymath}
                we have $ \sup_{i\geqslant0}\E_i^{(\widetilde{Q})}\mkern-1.5mu<\infty$. Construct an order-preserving conservative coupling $Q$-matrix $\overline{Q}=\bigl(\overline{q}(i,j;i^\prime,j^\prime)\bigr)$, whose marginalities are $Q$ and $\widetilde{Q}$, with non-diagonal entries
    				\begin{numcases}
    					{\overline{q}(i,j;i^\prime,j^\prime)=}
    						(i+1)\land(j+1), \quad&if\/ $i^\prime=i+1,\enspace i\geqslant0$,\nonumber\\
    											&$j^\prime=j+1,\enspace i_k<j<i_{k+1}$ for some $k\geqslant0$,\nonumber \\
    						(i+1)\land(j+1), &if\/ $i^\prime=i+1,\enspace i\geqslant0$,\nonumber\\
    											&$j^\prime=i_{k+1},\enspace j=i_k$ for some $k\geqslant0$,\nonumber \\
    						(i-j)^+,   &if\/ $i^\prime=i+1,\enspace i\geqslant0,\enspace j^\prime=j\geqslant0$,\nonumber\\
    						c, &if\/ $i^\prime=0,\enspace i\geqslant1,\enspace j^\prime=0,\enspace j=i_k$ for some $k\geqslant1$,\nonumber \\
                            \alpha_i-c, &if\/ $i^\prime=0,\enspace i\geqslant1,\enspace j^\prime=j=i_k$ for some $k\geqslant1$,\nonumber \\
                            \alpha_i, &if\/ $i^\prime=0,\enspace i\geqslant1,\enspace i_k<j^\prime=j<i_{k+1}$ for some $k\geqslant0$,\nonumber \\
    						0, & other $(i^\prime,j^\prime)\neq(i,j)$\nonumber.
    				\end{numcases}
                Denote the $\overline{Q}$-process as $\bigl(X(t),Y(t)\bigr)_{t\geqslant0}$, then we easily deduce that
    				\begin{displaymath}
                        \Psub{( i_1, i_2)}^{(\overline{Q})}\mkern-1.5mu\bigl[X(t)\leqslant Y(t)\bigr]=1,\qquad t>0,\enspace i_1\leqslant i_2.
    				\end{displaymath}
    			Hence,
    			\begin{displaymath}
                    \sup_{i\geqslant 1}\E_i^{(Q)}\mkern-1.5mu\sigma_0 \leqslant \sup_{i\geqslant 1}\E_i^{(\widetilde{Q})}\mkern-1.5mu\sigma_0 <\infty,
    			\end{displaymath}
    			so $Q$-process is strongly ergodic.
    		\end{proof}
\section{Applications to multi-dimensional examples}\label{chp_app}
    In this section, we shall apply our inverse problem criteria to some multi-dimensional models.
        Brussel's model (see \cite{yanchen1986}) is a typical model of reaction-diffusion process with several species.
		\begin{example}\label{brus}
            Let $S$ be a finite set, $E = (\Z_+^2)^S$ and let $p_k(u, v)$ be transition probability on $S$, $k = 1,2$. Denote by $e_{u 1} \in E$ the unit vector whose first component at site $u \in S$ is equal to 1 and the second component at $u$ as well as other components at $v \neq u $ all equal 0. Similarly, one can define $e_{u2}$. The model is described by the conservative $Q$-matrix $Q=(q_{ij})$:
                \begin{numcases}
                    {q(x, y)=}
                    \lambda_1a(u), \nonumber &if\/ $y=x+e_{u1}$, \\
                    \lambda_2b(u)x_1(u), \nonumber &if\/ $y=x-e_{u1}+e_{u2}$, \\
                    \lambda_3\binom{x_1(u)}{2}x_2(u), \nonumber &if\/ $y=x+e_{u1}-e_{u2}$, \\
                    \lambda_4x_1(u), \nonumber &if\/ $y=x-e_{u1}$, \\
                    x_k(u)p_k(u,v), \nonumber &if\/ $y=x-e_{uk}+e_{vk},\enspace k=1,2,\enspace v\neq u$,\\
                    0, & other $y\neq x$\nonumber,
                \end{numcases}
            and $q(x) = -q(x,x) =\sum_{y\neq x}q(x,y)$, where ${x = \Bigl(\bigl(x_1(u),x_2(u)\bigr)\st u\in S\Bigr)\in E}$. $a$ and $b$ are positive functions on $S$ and $\lambda_1,\ldots, \lambda_4$ are positive constants. Finite-dimensional Brussel's model is exponentially ergodic
            (cf.\@ \cite{chenjw1995}). We now demonstrate that it is non-strongly ergodic, which was  actually proved for the first time in \cite{wu2007}. But here we adopt different methods.
        \end{example}
        \begin{proof}
            We shall prove our assertion by two approaches. For ease of notation, we write $\widetilde{a}=\sum_{u\in S}a(u)$, $\abs{x}=\sum_{u\in S}\bigl(x_1(u)+x_2(u)\bigr)$ for $x\in E$ and also $E_i=\bigl\{x\in E\st\abs{x}=i\bigr\}$ for $i\geqslant 0$.
            \par
                    a) For each fixed $n\geqslant 1$, we construct function
                        \begin{displaymath}
                            F^{(n)}(x)=f^{(n)}_i,\qquad x\in E_i,\enspace i\geqslant 1,
                        \end{displaymath}
                    with
                        \begin{numcases}
                            {f^{(n)}_i=}
                            \frac{1}{\lambda_4}\log(i+1), \nonumber\quad &$1\leqslant i \leqslant n$, \\
                            \frac{1}{\lambda_4}\log(n+1), &$i\geqslant n+1$.\nonumber
                        \end{numcases}
                    Because
                        \begin{displaymath}\begin{split}
                            &\Bigl(1+\frac{1}{k}\Bigr)^{\ell}\leqslant \ue,\qquad 1\leqslant \ell\leqslant k,\\
                            \Bigl(1+\frac{1}{k}\Bigr)^{\ell}&\Bigl(\frac{k+1}{k+2}\Bigr)^{\frac{\lambda_1\widetilde{a}}{\lambda_4}}\leqslant \ue,\qquad 1\leqslant \ell\leqslant k,
                        \end{split}\end{displaymath}
                    we have
                        \begin{displaymath}\begin{split}
                            (\lambda_1\widetilde{a}+\lambda_4\ell)\frac{1}{\lambda_4}\log(k+1) &\leqslant\frac{\lambda_1\widetilde{a}}{\lambda_4}\log(k+1)+\ell\log k+1,\qquad 1\leqslant \ell\leqslant k,\\
                            (\lambda_1\widetilde{a}+\lambda_4\ell)\frac{1}{\lambda_4}\log(k+1) &\leqslant\frac{\lambda_1\widetilde{a}}{\lambda_4}\log(k+2)+\ell\log k+1,\qquad 1\leqslant \ell\leqslant k.
                        \end{split}\end{displaymath}
                    Now it is straightforward to check that
                        \begin{displaymath}\begin{split}
                            \Bigl(\lambda_1\widetilde{a}+\lambda_4\sum_{u\in S}x_1(u)\Bigr)f^{(n)}_i\leqslant \lambda_1\widetilde{a}f^{(n)}_{i+1}+\lambda_4\sum_{u\in S}x_1(u)f^{(n)}_{i-1}+1&,\\x\in E_i&,\enspace i\geqslant 1,\enspace n\geqslant 1,
                        \end{split}\end{displaymath}
                    where we naturally put $f^{(n)}_0=0\,(n\geqslant1)$.
            \par
                    It can be easily seen that $F^{(n)}(x)$ satisfies \Cref{in_serg_con_eq} in current setup and $\{F^{(n)}\}^{\infty}_{n=1}$ is a sequence satisfying conditions in \Cref{in_serg_con}. Consequently, we infer that finite-dimensional Brussel's model is non-strongly ergodic.
             \par
                    b) We try invoking \Cref{in_serg_con} yet with a different testing sequence. For each fixed $n\geqslant 1$, we construct function
                        \begin{displaymath}
                            F^{(n)}(x)=\sum^k_{i=1} d^{(n)}_i,\qquad x\in E_k,\enspace k\geqslant 1,
                        \end{displaymath}
                    with
                        \begin{numcases}
                            {d^{(n)}_i=}
                            \frac{1}{\lambda_4(i+1)}, \nonumber &$1\leqslant i \leqslant n$, \\
                            -\frac{1}{\lambda_1\widetilde{a}(n+1)}, \nonumber &$i=n+1$, \\
                            -\frac{1}{\lambda_1\widetilde{a}}, &$i\geqslant n+2$.\nonumber
                        \end{numcases}
                    \vskip -0.3 cm
                    Then a trivial calculation shows that $\{F^{(n)}\}^{\infty}_{n=1}$ is a sequence satisfying conditions in \Cref{in_serg_con}. So finite-dimensional Brussel's model is non-strongly ergodic.
        \end{proof}

		\begin{example}
            Let $E = \Z_+^2$. Epidemic process is defined by $Q$-matrix\break
            $Q = \Bigl(q\bigl((m, n), (m', n')\bigr)\st  (m,n), (m', n') \in E\Bigr)$ with
                \begin{numcases}
                    {q\bigl((m, n), (m', n')\bigr)=}
                    \alpha, \nonumber &if\/ $(m', n')= (m, n)$, \\
                    \gamma m, \nonumber &if\/ $(m', n')= (m, n)$, \\
                    \beta, \nonumber &if\/ $(m', n')= (m, n)$, \\
                    \delta n, \nonumber &if\/ $(m', n')= (m, n)$, \\
                    \varepsilon mn, \nonumber &if\/ $(m', n')= (m, n)$, \\
                    0, & otherwise, unless $(m', n')=(m, n)$\nonumber,
                \end{numcases}
            and $q(m, n) = -q\bigl((m,n),(m,n)\bigr)= \sum_{(m^\prime,n^\prime)\neq (m,n)}q\bigl((m, n), (m', n')\bigr)$, where $\alpha,\gamma,\beta,\delta$, and $\varepsilon$ are non-negative constants. We assume $\gamma>0$ and $\delta>0$. The $Q$-process is unique and ergodic when $\alpha+\beta>0$ (cf.\@ \cite{anderson}). Epidemic process is non-strongly ergodic if\/ $\alpha+\beta$, $\gamma$, and $\delta$ are strictly positive by \cite{wu2007}. Using similar argument as in $\Cref{brus}$, we can also carry out this result and therefore give a new proof. We will not reproduce the details here.
        \end{example}
		
        \begin{example}\label{gamma}
            Consider a conservative birth-death $Q$-matrix with birth rate $b_0=1$, $b_i=i^{\gamma}\,(i\geqslant 1)$ and death rate $a_i=i^{\gamma}\,(i\geqslant 1)$. It is known that this $Q$-matrix is regular for all $\gamma\in\R$ and the $Q$-process is recurrent. The process is ergodic iff $\gamma >1$ and strongly ergodic iff $\gamma>2$ (cf. \cite{chen2004}). We now use \Cref{in_serg_con} to demonstrate that the process is non-strongly ergodic if $\gamma\leqslant 2$. Also, we use \Cref{in_erg_con} to present that the process is non-ergodic if $\gamma\leqslant 1$.
        \end{example}
        \begin{proof}
                    a) First we prove the process is non-strongly ergodic if $\gamma\leqslant2$ using \Cref{in_serg_con}. For each fixed $n\geqslant 1$, define
                        \begin{displaymath}
                            y^{(n)}_k=\sum^k_{i=1} d^{(n)}_i, \qquad k\geqslant 1,
                        \end{displaymath}
                        with
                        \begin{numcases}
                            {d^{(n)}_i=}
                            \frac{1}{i^{1+\frac{1}{i}}},\quad  &$1\leqslant i \leqslant n$, \nonumber\\
                            \frac{1}{i^{1+\frac{1}{n+1}}},  &$i\geqslant n+1$ \nonumber.
                        \end{numcases}
                    When $\gamma\leqslant 2$, we have the following estimates:
                        \begin{subequations}\label{calc}\begin{align}
                            \frac{1}{i^{1+\frac{1}{i}}}-\frac{1}{(i+1)^{1+\frac{1}{i+1}}}&\leqslant \frac{1}{i^{\gamma}},\qquad i\geqslant 1\label{calc1},\\
                            \frac{1}{i^{1+\frac{1}{n+1}}}-\frac{1}{(i+1)^{1+\frac{1}{n+1}}}&\leqslant \frac{1}{i^{\gamma}},\qquad i\geqslant n+1\label{calc2}.
                        \end{align}\end{subequations}
                    \par
                    In fact, \Cref{calc1} holds obviously for $i=1,2$. Put
						\begin{displaymath}
							g_1(x)=\frac{1}{x^{1+\frac{1}{x}}},\qquad x>0.
						\end{displaymath}
					Differentiating $g_1$, we obtain
                        \begin{displaymath}
                            \abs{ g_1^\prime(x)}=\frac{1}{x^{2+\frac{1}{x}}}\Bigl(1+\frac{1-\log x}{x}\Bigr)
                            \leqslant \frac{1}{x^{2}}
                            \leqslant \frac{1}{x^{\gamma}},\qquad \text{if } x\geqslant \ue.
                        \end{displaymath}
                    By Lagrange mean value theorem, \Cref{calc1} holds.
                    \par
                    We turn to \Cref{calc2}. Denote $\varepsilon=\frac{1}{n+1}$, then we have
                        \begin{displaymath}\begin{split}
                            \frac{1}{i^{1+\varepsilon}}-\frac{1}{(i+1)^{1+\varepsilon}}&=\frac{(i+1)^{1+\varepsilon}-i^{1+\varepsilon}}{i^{1+\varepsilon}(i+1)^{1+\varepsilon}}
                            \leqslant \frac{(1+\varepsilon)(i+1)^{\varepsilon}}{i^{1+\varepsilon}{(i+1)}^{1+\varepsilon}}\\
                            &=\frac{1+\varepsilon}{i^{1+\varepsilon}(i+1)}= \frac{1}{i^{2}}(1+\varepsilon)\frac{i^{1-\varepsilon}}{i+1},
                        \end{split}\end{displaymath}
                    where ``$\leqslant$'' is obtained by mean value theorem.\\
                    Define
                        \begin{displaymath}
                            g_2(x)=(1+\varepsilon)\frac{x^{1-\varepsilon}}{x+1},\qquad x>0.
                        \end{displaymath}
                    By calculus method, we see that $g_2$ is decreasing on the interval $[n+1,\infty)$. One can also verify easily that $g_2(n+1)\leqslant 1$. Therefore
                        \begin{displaymath}
                            g_2(i)=(1+\varepsilon)\frac{i^{1-\varepsilon}}{i+1}\leqslant 1,\qquad i\geqslant n+1.
                        \end{displaymath}
                    And \Cref{calc2} follows.\\
                    By \Cref{calc}, $\bigl(y^{(n)}_i\bigr)_{i\geqslant 1}$ satisfies \Cref{in_serg_con_eq} in current setup:
                        \begin{equation}\label{gameq}
                            d^{(n)}_i\leqslant d^{(n)}_{i+1}+\frac{1}{i^\gamma},\qquad i\geqslant1.
                        \end{equation}
                    and $\{y^{(n)}\}^{\infty}_{n=1}$ is a sequence satisfying all conditions in \Cref{in_serg_con}. Consequently, we conclude that the $Q$-process is non-strongly ergodic if $\gamma\leqslant 2$.
					\par
                    b) We use \Cref{in_serg_con} to deduce non-strong ergodicity yet with a different testing sequence. Define
                        \begin{numcases}
                            {d^{(n)}_i=}
                            \frac{1}{(i+9)\log(i+9)},  &$1\leqslant i \leqslant n$, \nonumber\\
                            \frac{1}{(n+9)\log(n+9)}-\sum^{i-1}_{k=n}\frac{1}{k^2},  \quad&$i\geqslant n+1$\nonumber.
                        \end{numcases}
                    Because
                    \begin{displaymath}\begin{split}
                            &\sum^{\infty}_{k=n}\frac{1}{k^2}>\int^{\infty}_{n}\frac{1}{x^2}\df{x} =\frac{1}{n}>\frac{1}{(n+9)\log(n+9)},\quad n\geqslant 1,\\
                            &\frac{1}{(i+9)\log(i+9)}-\frac{1}{(i+10)\log(i+10)}\leqslant \frac{1}{i^2},\quad i\geqslant 1,
                    \end{split}\end{displaymath}
                    it is straightforward to verify that $\{y^{(n)}\}^{\infty}_{n=1}$, with
                    $ y^{(n)}_k=\sum^k_{i=1} d^{(n)}_i\,(k\geqslant 1)$, is a sequence satisfying conditions in \Cref{in_serg_con}.
					\par
                    c) We now turn to non-ergodicity. For each $n\geqslant1$, we set
                        \begin{displaymath}
                            y^{(n)}_0=n+1,\quad  y^{(n)}_i=\sum^i_{k=1}d^{(n)}_k\,(i\geqslant1),
                        \end{displaymath}
                    where $d^{(n)}_k=n-\sum_{j=1}^{k-1}\frac{1}{j}\,(k\geqslant1)$ and $\sum_{\varnothing}=0$.
                    Hence for each $n\geqslant 1$, $\bigl(y^{(n)}_i\bigr)_{i\geqslant0}$ satisfies
                        \begin{displaymath}\begin{split}
                                y^{(n)}_0&\leqslant y^{(n)}_1+1,\\
                                d^{(n)}_i&\leqslant d^{(n)}_{i+1}+\frac{1}{i^\gamma},\qquad i\geqslant1,
                        \end{split}\end{displaymath}
                    which is exactly \Cref{in_erg_con_eq} in current setup.
                    So, $\{y^{(n)}\}^{\infty}_{n=1}$, with $y^{(n)}=\bigl(y^{(n)}_i\bigr)_{i\geqslant0}$, is a sequence for \Cref{in_erg_con}. Therefore, the $Q$-process is non-ergodic for $\gamma\leqslant1$.
        \end{proof}
        \par
        We further investigate a multi-dimensional version of \Cref{gamma}.
        \renewcommand{\thetheorem}{\ref{gamma}$^\prime$}
        \addtocounter{theorem}{-1}
        \begin{example}\label{gammamult}
             Let $S$ be a finite set, $E = (\Z_+)^S$ and $p(u, v)$ a transition probability matrix on $S$. We denote by $\theta \in E$ whose components are identically 0 and denote by $e_u \in E$ the unit vector whose component at site $u \in S$ is equal to 1 and other components at $v \neq u $ all equal 0. Define an irreducible $Q$-matrix $Q=\bigl(q(x,y)\st x,y\in E\bigr)$ as follows:
                \begin{numcases}
                    {q(x, y)=}
                    x(u)^{\gamma}, \nonumber &if\/ $y=x+e_u$,\enspace $x\neq\theta$, \\
                    1, \nonumber &if\/ $x=\theta$,\enspace $y=e_u$, \\
                    x(u)^{\gamma}, \nonumber &if\/ $y=x-e_u$, \\
                    x(u)p(u,v), \quad\nonumber &if\/ $y=x-e_u+e_v$,\enspace $v\neq u$,\\
                    0, & other $y\neq x$\nonumber,
                \end{numcases}
            and $ q(x) = -q(x,x)=\sum_{y\neq x}q(x,y)$, where $x = \bigl(x(u)\st u\in S\bigr) \in E$. It is easy to check by \cite[Theorem 1]{yanchen1986} that the $Q$-process is unique for all $\gamma\in \R$. We now prove the following results:
                \begin{enumerate}[(1)]
                    \item When $\gamma\leqslant 2$, the $Q$-process is non-strongly ergodic.
                    \item When $\gamma\leqslant 1$, the $Q$-process is non-ergodic.
                \end{enumerate}
        \end{example}
        \renewcommand{\thetheorem}{\arabic{theorem}}
        \begin{proof}
            We will reduce multi-dimensional problem to 1-dimensional case. We write $\abs{x}=\sum_{u\in S}x(u)$ for $x\in E$ and $E_i=\bigl\{x\in E\st\abs{x}=i\bigr\}$ for $i\geqslant 0$.
			\par
            a) Using \Cref{in_serg_con}, to prove that the $Q$-process is non-strongly ergodic for $\gamma\leqslant2$, we need only construct sequence $\{F^{(n)}\}^{\infty}_{n=1}$ satisfying the conditions. We may guess $F^{(n)}$ is identically $f^{(n)}_i$ on $E_i$ for each $i\geqslant1$, and set
                \begin{displaymath}
                    d^{(n)}_1=f^{(n)}_1,\quad
                    d^{(n)}_i=f^{(n)}_i-f^{(n)}_{i-1}\,(i\geqslant 2).
                \end{displaymath}
            Now, \Cref{in_serg_con_eq} becomes
				\begin{displaymath}
                    d^{(n)}_i\leqslant d^{(n)}_{i+1}+\frac{1}{\sum_{u\in S}x(u)^{\gamma}},\qquad x\in E_i,\enspace i\geqslant 1.
				\end{displaymath}
			Because
				\begin{displaymath}
                    \sum_{u\in S}x(u)^{\gamma}\leqslant \sum_{u\in S}x(u)^2\leqslant \Bigl(\sum_{u\in S}x(u)\Bigr)^2=i^2,\qquad x\in E_i,\enspace i\geqslant 1,\enspace \gamma\leqslant2,
				\end{displaymath}
			we need only construct sequence satisfying
                \begin{displaymath}
                    d^{(n)}_i\leqslant d^{(n)}_{i+1}+\frac{1}{i^2},\qquad i\geqslant1,
                \end{displaymath}
            which is exactly \Cref{gameq} with $\gamma=2$. Now we can proceed our proof as in \Cref{gamma}. The $Q$-process is therefore non-strongly ergodic if $\gamma\leqslant2$.
			\par
            b) To deal with non-ergodicity, according to the discussions in a) and using similar notations, we need only consider equation
                 \begin{displaymath}\begin{split}
                    y_0&\leqslant y_1+1,\\
                    d_i&\leqslant d_{i+1}+\frac{1}{i},\qquad i\geqslant1.
                 \end{split}\end{displaymath}
                And we can proceed as in proof c) of \Cref{gamma}. Hence the multi-dimensional process is non-ergodic for $\gamma\leqslant1$.
        \end{proof}

\noindent {\bf Acknowledgement:} Thanks to Prof.\@ Mu-Fa Chen for his careful guidance and valuable suggestions. This work is supported by the National Nature Science Foundation of China (Grant No.\@ 11771046).

\bibliographystyle{plain}
\bibliography{ip}

\begin{thebibliography}{10}

\bibitem{anderson}
William~J Anderson.
\newblock {\em Continuous-Time Markov Chains}.
\newblock Springer, New York, 1991.

\bibitem{chen2004}
Mu-Fa Chen.
\newblock {\em From Markov Chains to Non-Equilibrium Particle Systems}.
\newblock World Scientific, Singapore, 2nd edition, 2004.

\bibitem{wang2003}
Mu-Fa Chen and Ying-Zhe Wang.
\newblock Algebraic convergence of markov chains.
\newblock {\em The Annals of Applied Probability}, 13(2):604--627, May 2003.

\bibitem{chenzhang2014}
Mu-Fa Chen and Yu-Hui Zhang.
\newblock Unified representation of formulas for single birth processes.
\newblock {\em Front. Math. China}, 9(4):761--796, 2014.

\bibitem{martin2016}
Martin Hairer.
\newblock {\em Convergence of Markov Processes, Lecture Notes}.
\newblock http://www.hairer.org/notes/Convergence.pdf, 01 2016.

\bibitem{hou1988}
Zhen-Ting Hou and Qing-Feng Guo.
\newblock {\em Time-homogeneous Countable Markov Processes}.
\newblock Science Press, 1978.

\bibitem{mao2003}
Yong-Hua Mao.
\newblock Algebraic convergence for discrete-time ergodic markov chains.
\newblock {\em Science in China (Ser. A)}, 46(5):621--630, 2003.

\bibitem{mao2004}
Yong-Hua Mao.
\newblock Ergodic degrees for continuous-time markov chains.
\newblock {\em Science in China Ser. A Mathematics}, 47(2):161--174, 2004.

\bibitem{twd1981}
Richard~Lewis Tweedie.
\newblock Criteria for ergodicity, exponential ergodicity and strong ergodicity
  of markov processes.
\newblock {\em J. Appl. Prob.}, 18(1):122--130, 1981.

\bibitem{wu2007}
Bo~Wu and Yu-Hui Zhang.
\newblock A class of multidimensional $\textit{Q}$-processes.
\newblock {\em J. Appl. Prob.}, 44(1):226--237, 2007.

\bibitem{yanchen1986}
Shi-Jian Yan and Mu-Fa Chen.
\newblock Multi-dimensional $\textit{Q}$-processes.
\newblock {\em Chinese Ann. Math.}, 7B(1):90--110, 1986.

\bibitem{zhang2001}
Yu-Hui Zhang.
\newblock Strong ergodicity for single-birth processes.
\newblock {\em J. Appl. Prob.}, 38(1):270--277, 2001.

\end{thebibliography}

\end{document}